\documentclass{article}
\usepackage[margin= 0.8in]{geometry}
\usepackage{listings}
\usepackage{graphicx}
\usepackage{amsmath}
\usepackage{amsthm}
\usepackage{amssymb}
\usepackage{amsfonts}
\usepackage{latexsym}
\usepackage{mathrsfs}
\usepackage{epsfig}
\usepackage{cite}
\usepackage[active]{srcltx}
\usepackage{cmap}
\usepackage{curves}
\usepackage{color}
\usepackage{subfigure}
\usepackage[dvipsnames]{xcolor}
\definecolor{dark}{gray}{0.4}
\usepackage{fancybox}
\usepackage{array}
\usepackage{pdfpages}
\usepackage{listings}
\usepackage{bbm}
\usepackage{tikz}
\usetikzlibrary{shapes.geometric,positioning}
\usepackage{mathtools,xparse}
\usepackage{algorithm}
\usepackage{algorithmic}
\usepackage{amsthm}
\theoremstyle{definition}
\newtheorem{definition}{Definition}

\newtheorem{Proposition}{Proposition}
\newtheorem{Theorem}{Theorem}
\newtheorem{Corollary}{Corollary}
\newtheorem{assumption}{Assumption}
\newtheorem{Lemma}{Lemma}

\DeclarePairedDelimiter{\norm}{\lVert}{\rVert}
\NewDocumentCommand{\normL}{ s O{} m }{%
  \IfBooleanTF{#1}{\norm*{#3}}{\norm[#2]{#3}}_{L_2(\Omega)}%
}

\title{ Unbiased Simulation for Optimizing Stochastic Function Compositions}
\author{Jose Blanchet, Donald Goldfarb, Garud Iyengar, Fengpei Li, Chaoxu Zhou}
\date{}
\begin{document}
\maketitle
\begin{abstract}
In this paper, we introduce an unbiased gradient simulation algorithms for solving convex optimization problem with stochastic function compositions. We show that the unbiased gradient generated from the algorithm has finite variance and finite expected computation cost. We then combined the unbiased gradient simulation with two variance reduced algorithms (namely SVRG and SCSG) and showed that the proposed optimization algorithms based on unbiased gradient simulations exhibit satisfactory convergence properties. Specifically, in the SVRG case, the algorithm with simulated gradient can be shown to converge linearly to optima in expectation and almost surely under strong convexity. Finally, for the numerical experiment,we applied the algorithms to two important cases  of stochastic function compositions optimization:  maximizing the Cox's partial likelihood model and training conditional random fields.
\end{abstract}

\section{Introduction}

\subsection{Motivation}

In machine learning, we often encounter the following optimization problem. Let  $f_1,...,f_n$ be a sequence of vector functions from $\mathbb{R}^d$ to $\mathbb{R}$. Our goal is to find an approximate solution of the following optimization problem, also known as the empricial risk minimization (ERM) problem,
\begin{equation}
	\min\limits_{x} F(x), \qquad F(x) \triangleq \frac{1}{n} \sum_{i=1}^{n}f_{i}(x)
\end{equation}
The standard method of SGD can be described by the following update rule for $t=1,2,...$
\begin{equation}
	x^{(t)}=x^{(t-1)}-\lambda_t(\nabla  f_{v_t}(x_{t-1})) \qquad \text{and} \qquad \mathbb{E}[x^{(t)}|x^{(t-1)}]=x^{(t-1)}-\frac{\lambda_t}{n}\sum_{i=1}^n \nabla f_{i}(x^{(t-1)})
\end{equation}

where $v_t$ follows uniform distribution on $\{1,2,...,n\}$. Stochastic gradient descent (SGD) and its variance reduced variants including SVRG have been shown to be powerful tools for solving the ERM problem, when $n$ is large and computing the full gradient is computationally intensive. However, most of these algorithms implicitly assume that the gradient of each member function $f_i(\cdot)$, $i = 1, \ldots, n$ is easy to obtain. However, this assumption fails to hold in the stochastic composition optimization problems \cite{wang2017stochastic}
\begin{align}
	\min_{x \in \mathbb{R}^p} F(x) \triangleq \frac{1}{n}\sum_{i=1}^n f_i(\frac{1}{m}\sum_{j=1}^m g_m(x)).
\end{align}
 where $v$ and $w$ follows certain distribution. Problem of this form arises in many areas such as reinforcement learning, risk-averse learning to graphical model, econometrics and survival analysis. The current algorithms used to solve this problem are based on \textit{biased} stochastic gradient oracles. As we know, the convergence rates for these algorithms  are either unsatisfactory compared to  generic stochastic optimization algorithms or heavily dependent on the number of component functions $m$ and $n$. To overcome these drawbacks, we introduce a couple of variance reduced algorithms that involve the simulation of unbiased stochastic gradients via the Multilevel Monte-Carlo.  


\subsection{Contributions}
The contribution of this paper is two-folded. First, we introduce unbiased gradient simulation algorithms for solving stochastic composition optimization problem. With an unbiased gradient simulation procedure, the stochastic composition optimization problem can be reduced to a generic stochastic optimization problem. We also construct a unbiased gradient simulation algorithm to take advantage of the finite sum structure. We also show that the computational cost of the unbiased gradient simulation algorithms is independent of the dimension of the objection function. Secondly, we apply our algorithms to maximize the partial likelihood function in Cox's model whose computational issues has not been fully addressed in machine learning literature so far. Specifically, when the sample size is large, solving this problem is known to be a computationally intensive task because of the cumulative sum structure that involves all data in the risk set presents in every component function. Our unbiased gradient simulation algorithms provide an efficient way to collapse the cumulative sum structure and the variance reduced gradient methods could further boost the rate of convergence. 

\subsection{Related works}

In the stochastic composition optimization literature, all algorithms are based on biased stochastic gradient. \cite{wang2017stochastic} first proposed a generic algorithm for solving \eqref{prob:composition_stoc} with a convergence rate $O(k^{-1/4})$ for convex objectives and $O(K^{-2/3})$ for strongly convex objectives. This result has been improved to $O(k^{-4/5})$ for strongly convex objectives by \cite{wang2016accelerating}. Recently, \cite{pmlr-v54-lian17a} further improves the convergence rate to $O(\rho^{K/(m+n + \kappa^2)})$ for the finite sum problem \eqref{prob:composition_finite} by utilizing a stochastic variance reduced gradient algorithm (SVRG). However, in this paper, we proposed a unbiased gradient simulation method that combines recent development in \cite{rhee2015,blanchet2015unbiased}. In particular, we employ the methods proposed in \cite{blanchet2015unbiased} which combines a bias removal randomization scheme into Multilevel Monte Carlo method proposed in \cite{rhee2015}. We then further make use of the  SVRG \cite{johnson2013accelerating} algorithm which can greatly reduce variance for ERM problem that achieves linear convergence. SVRG has been extended and improved in many works including but not limited to \cite{xiao2014proximal}, \cite{allen2016improved}, \cite{harikandeh2015stopwasting}, \cite{lei2017less}, \cite{gong2014linear}, \cite{nitanda2014stochastic}. SAG \cite{schmidt2013minimizing} and SAGA \cite{defazio2014saga} are two examples of incremental gradient methods that achieve linear convergence.  

\subsection{Organization}
In section 2, we will give some concrete examples that is formulated as \eqref{prob:composition_stoc} and \eqref{prob:composition_finite}. In section 3, we will describe our unbiased gradient simulation algorithms for the stochastic problem \eqref{prob:composition_stoc} and the finite sum problem \eqref{prob:composition_finite}. Then based on these two algorithms, we present the algorithms for both problems. In section 4, we will first show that the gradients generated by the simulation algorithms are unbiased, has finite second moments and the expected computation cost is finite. Finally we will show that our variance reduced  algorithms converges linearly to an $\epsilon$-approximated solution in expectation for both problems. In section 5, we implement our algorithms for maximizing the Cox's partial likelihood and present our numerical results. We concludes with remarks on possible future work.

\section{Problem Description and Algorithms}
\subsection{Problem description}
Through out this paper, we consider numerical solutions of the stochastic optimization problem below
\begin{align}
	\min_{x \in \mathbb{R}^p} F(x) \triangleq \mathbb{E}_vf_v(\mathbb{E}_w g_w(x)). \label{prob:composition_stoc}
\end{align}
Note that the following two problems can be considered as special cases of \eqref{prob:composition_stoc}, the finite sum problem
\begin{align}
	\min_{x \in \mathbb{R}^p} F_n(x) \triangleq \frac{1}{n} \sum_{i = 1}^n f_i(\frac{1}{m_i}\sum_{j = 1}^{m_i} g_j(x)), \label{prob:composition_finite}
\end{align}
or the mixed problem 
\begin{align}
\min_{x \in \mathbb{R}^p} \frac{1}{n} \sum_{i = 1}^n f_i(\mathbb{E}_w g(x, w)). \label{prob:composition_mix}
\end{align}
We will also discuss numerical algorithms for these two special cases. We assume $f_v:\mathbb{R}^d \rightarrow \mathbb{R}$ is $\mu$-strongly convex and has $L$-Lipschitz gradients, for each component $v$ and $g_w:\mathbb{R}^p \rightarrow \mathbb{R}^d$ for each component $w$. The gradient (with respect to $x$) of each member function $f_v(\cdot)$ for the stochastic problem is $\{\mathbb{E}_w
\nabla g_w(x)\}\nabla f_v\{\mathbb{E}_w
 g_w(x)\}$, so that 
 \begin{equation}
 \nabla F_n(x) = \{\mathbb{E}_w \nabla g_w(x)\}^\intercal\mathbb{E}_v\{\nabla f_v(\mathbb{E}_w g_w(x))\}
 \end{equation}
 where 
\begin{align*}
\nabla g_w(x) = \left(\begin{array}{ccc}
\frac{\partial [g_w]_1}{\partial [x]_1}(x) & \cdots & \frac{\partial [g_w]_1}{\partial [x]_p}(x) \\
\vdots & \ddots & \vdots \\
\frac{\partial [g_w]_d}{\partial [x]_1}(x) & \cdots & \frac{\partial [g_w]_d}{\partial [x]_p}(x)
\end{array}\right),
\end{align*}
and 
\begin{align*}
	g_w(x) = ([g_w]_1(x), [g_w]_2(x), \cdots [g_w]_d(x))^\top.
\end{align*}
\subsection{Unbiased stochastic gradient simulation}
We present the algorithm to simulate unbiased gradients for the stochastic problem \eqref{prob:composition_stoc}, \eqref{prob:composition_finite} and \eqref{prob:composition_mix}. They can be considered as variants of \cite{blanchet2015unbiased} which is based on multi-level randomization technique. In the first algorithm we purpose for simulating unbiased gradient for problem \eqref{prob:composition_stoc} and \eqref{prob:composition_mix} while fixing a component $v_1$ for $ f_{v_1}(\mathbb{E}_w g(x, w))$.  The base level $n_0$ of estimator can be raised to reduce variance. We introduce a couple of notations first. 
\begin{definition}\label{sandt}
Fix $x \in \mathbb{R}^p$, we define $S(x)= \nabla g_{w}(x) \in \mathbb{R}^{d \times p}$, $T(x)=g_{w}(x) \in \mathbb{R}^d$ and $Z(x)=\nabla^2 g_{w}(x) \in \mathbb{R}^{d\times p \times p}$ where $w$ is random. Specifically, sample I.I.D $\{w_i\}_{i\geq 1}$ from the distribution of $w$, we define $S_i(x)= \nabla g_{w_i}(x)$, $T_i(x)=g_{w_i}(x)$ and $Z_i(x)=\nabla^2 g_{w_i}(x)$. Also, we write
 $\bar{S}_n(x)=\frac{1}{n}\sum\limits_{i=1}^n S_i(x)$, $ \tilde{S}_n(x)=\frac{1}{n}\sum\limits_{i=n+1}^{2n} S_i(x)$ and similarly for $\bar{T}_n(x),\tilde{T}_n(x),\bar{Z}_n(x),\tilde{Z}_n(x)$. It follows that, for any $n$,
 \begin{equation}\label{SandT}
\bar{S}_{2n}(x)=\frac{1}{2}(\bar{S}_n(x)+\tilde{S}_n(x)) \quad \text{,} \quad \bar{T}_{2n}(x)=\frac{1}{2}(\bar{T}_n(x)+\tilde{T}_n(x)) \quad \text{and} \quad \bar{Z}_{2n}(x)=\frac{1}{2}(\bar{Z}_n(x)+\tilde{Z}_n(x))
\end{equation}

\end{definition}
\begin{algorithm}[htb]
	\caption{UnbiasedGradient($x,v_1$)}
	\label{alg:mlmc2}
	\begin{algorithmic}
		\STATE\textbf{Input:} $x \in \mathbb{R}^p$,$v_1 \in \{1,...,n\}$, base level of estimator $n_0 \geq 0$, rate parameter $1 < \gamma < 2$.  \\
		\STATE \textbf{Output:} $W(x,v_1) \in \mathbb{R}^p$, an unbiased estimate of the gradient of $ f_{v_1}(\mathbb{E}_w g(x, w))$ at point $x$ and component $v_1$.
		\STATE Sample $N$ follow geometric distribution with success probability $1-p$ where $p=0.5^{\gamma}$. 
		\STATE Sample I.I.D. $\{w_i\}_{1 \leq i \leq 2^{N+n_0+1}}$ follow the distribution of $w$ and obtain $\{S_i(x),T_i(x)\}_{1\leq i\leq 2^{N+n_0+1}}$.
		\STATE Set $Y_1 = [\bar{S}_{2^{N+n_0+1}}(x)]^\intercal \cdot \nabla f_{v_1}(\bar{T}_{2^{N+n_0+1}}(x))$.  Set $Y_2 = [\bar{S}_{2^{N+n_0}}(x)]^\intercal \cdot \nabla f_{v_1}(\bar{T}_{2^{N+n_0}}(x))$.
		\STATE Set $Y_3 = [\tilde{S}_{2^{N+n_0}}(x)]^\intercal \cdot \nabla f_{v_1}(\tilde{T}_{2^{N+n_0}}(x))$. $\quad$ Set $Y_4 = [\bar{S}_{2^{n_0}}(x)]^\intercal \cdot \nabla f_{v_1}(\bar{T}_{2^{n_0}}(x))$.
		\STATE Set $W(x,v_1)=\frac{Y_1-0.5 \cdot (Y_2+Y_3)}{\tilde{p}_N} + Y_4$, where $\tilde{p}_N=(1-p)\cdot p^N$.
		\STATE \textbf{Output:} $W(x,v_1)$
	\end{algorithmic}
\end{algorithm}

We shall prove in section 4 that algorithm 1 outputs an unbiased estimate of $ f_{v_1}(\mathbb{E}_w g(x, w))$ for fixed $v_1$. It follows that if we sample $v_1\sim v$, then $W(x,v_1)$ would be an unibased estimate of the gradient of $\mathbb{E}_vf(\mathbb{E}_w g(x, w), v)$.   The algorithm 1 presented here is in its most general form which can be applied to unbiased gradient simulation for all three problems \eqref{prob:composition_stoc}, \eqref{prob:composition_mix} and \eqref{prob:composition_finite}. 
We also present another algorithm below tailored for the finite sum problem \eqref{prob:composition_finite} where $\mathbb{E}_vf(\mathbb{E}_w g(x, w), v)$ can be written as $\frac{1}{n} \sum_{i=1}^n f_i(\frac{1}{m_i}\sum_{j=1}^{m_i}g_j(x))$. The key change in algorithm 2 is to truncate the geometric random variable to take into account the case when the first algorithm requires more samples than the size of overall data. We discuss the details of these algorithms in section 4.

\begin{algorithm}[htb]
	\caption{Unbiased Estimator of Gradient for finite sum problems using Multilevel Monte-Carlo}
	\label{alg:mlmc}
	\begin{algorithmic}
		\STATE\textbf{Input:} $x \in \mathbb{R}^p$,$v_1 \in \{1, \ldots, n\}$, base level of estimator $n_0 \geq 0$, rate parameter $1 < \gamma < 2$.  \\
		\STATE \textbf{Output:} $W(x,v_1) \in \mathbb{R}^p$, an unbiased estimator of the gradient of $\frac{1}{n} \sum_{i=1}^n f_i(\frac{1}{m_i}\sum_{j=1}^{m_i}g_j(x))$ in \eqref{prob:composition_finite} at point $x$. 
		\STATE Sample $N$ follow geometric distribution with success probability $1 - p$, where $ p = 0.5^\gamma$. 
		\STATE Set $n_1 = \lfloor\log_2(m_{v_1})\rfloor$, $N_2$ = $N\mod (n_1-n_0+1)$ and $\tilde{p}_{N_2}=p^{N_2}(1-p)(1-p^{N_1-n_0+1})^{-1}$
		\IF {$n_0 \geq n_1$}
		\STATE Set $ W(x,v_1) = \{ \frac{1}{m_{v_1}} \sum_{j = 1}^{m_{v_1}} [\nabla g_j(x)]^\intercal\} \nabla f_{v_1}\{\frac{1}{m_{v_1}}\sum_{j=1}^{m_{v_1}}g_j(x)\}$
		\ELSIF{$n_2=n_1-n_0$}
		\STATE  Uniformly sample with replacement $\{w_i\}_{1\leq i \leq 2^{n_1} }$ from $\{1, \ldots, m_{v_1}\}$.
		\STATE Set $Y_1 = \{\frac{1}{m_{v_1}} \sum_{j = 1}^{m_{v_1}} [\nabla g_j(x)]^\intercal\} \nabla f_{v_1}\{\frac{1}{m_{v_1}}\sum_{j=1}^{m_{v_1}}g_j(x)\}$. 
		\STATE Set $Y_2 = \{\frac{1}{2^{n_1}}\sum_{i = 1}^{2^{n_1}}[\nabla g_{w_i}(x)]^\intercal\}\nabla f_{v_1} \{\frac{1}{2^{n_1}}\sum_{i = 1}^{2^{n_1}} g_{w_i}(x) \}$.
		\STATE Set $Y_3 = \{\frac{1}{2^{n_0}}\sum_{i = 1}^{2^{n_0}}[\nabla g_{w_i}(x)]^\intercal\}\nabla f_{v_1} \{\frac{1}{2^{n_0}}\sum_{i = 1}^{2^{n_0}} g_{w_i}(x) \}$
		\STATE Set $W(x,v_1)=\frac{Y_1-Y_2}{\tilde{p}_N} + Y_3$;
		\ELSE 
		\STATE Uniformly sample with replacement $\{w_i\}_{1\leq i \leq 2^{n_1+n_0+1} }$ from $\{1, \ldots, m_{v_1}\}$.
		\STATE Set $Y_1 = [\bar{S}_{2^{N_2+n_0+1}}(x)]^\intercal \cdot \nabla f_{v_1}(\bar{T}_{2^{N_2+n_0+1}}(x))$. Set $Y_2 = [\bar{S}_{2^{N+n_0}}(x)]^\intercal \cdot \nabla f_{v_1}(\bar{T}_{2^{N+n_0}}(x))$.
		\STATE Set $Y_3 = [\tilde{S}_{2^{N+n_0}}(x)]^\intercal \cdot \nabla f_{v_1}(\tilde{T}_{2^{N+n_0}}(x))$. $\qquad$ Set $Y_4=[\bar{S}_{2^{n_0}}(x)]^\intercal \cdot \nabla f_{v_1}(\bar{T}_{2^{n_0}}(x))$.
		\STATE Set $W(x,v_1)=\frac{Y_1-0.5 \cdot (Y_2+Y_3)}{\tilde{p}_N} + Y_4$
		\ENDIF
		\STATE \textbf{Output:} $W(x,v_1)$
	\end{algorithmic}
\end{algorithm}
\textbf{Remark}: In this algorithm, we truncated the geometric random variable $N$ at $n_1-n_0+1$ and adjust its probability mass function at $n_2$ from $p^{n_2}(1-p)$ to $\tilde{p}_{n_2}=p^{n_2}(1-p)(1-p^{n_1-n_0+1})^{-1}$ to account for the truncation.
\subsection{Optimization Algorithms}

  We now present our algorithms to solve problem \eqref{prob:composition_stoc}, \eqref{prob:composition_mix} and \eqref{prob:composition_finite}.
It is based on the unbiased gradient simulation algorithms just introduced as well as the control variate method for variance reduction. In \cite{JZ}, \cite{frostig2015competing}, the control variate methods ia used to generate variance reduced stochastic gradients for solving $\min_{x \in \mathbb{R}^p} \mathbb{E}_\xi f(x,v)$. For example, for a function of the form $F(x)=\frac{1}{n}\sum_{i=1}^n f_i(x)$, a variance reduced stochastic gradient at point $x$ with respect to the reference point $\tilde{x}$ is defined as $\nabla_x f(x, v_1) - \nabla_x f(\tilde{x}, v_1) + \mathbb{E}_v \nabla_x f(\tilde{x}, v)$ where $v_1$ is sampled from $v$. In contrast to SGD where the stochastic gradient is simply $\nabla_x f(x, v_1)$, the variance reduced algorithms use constant step size and converge linearly to the optimum in the presence of strong convexity. 

We adopt the variance reduction techniques into the current setting of optimizing function compositions with simulated unbiased gradients. Specifically, we simulate the unbiased gradients at $x$ and $\tilde{x}$ simultaneously, using the same set of data, to control variance. We summarize the details of generating variance reduced gradient in algorithm 3. The procedure in algorithm 3 is based on the setting of algorithm 1 for the ease of presentation and it can be modified to suit the improved algorithm 2 as well. 

\begin{algorithm}[htb]
	\caption{SimulatedGradient(x, $\tilde{x}$, $g(\tilde{x})$)}
	\label{alg:mlmc3}
	\begin{algorithmic}
		\STATE\textbf{Input:} $x \in \mathbb{R}^d $, $v_1 \in \Omega_v$, reference point $\tilde{x} \in \mathbb{R}^d$, reference gradient at point $\tilde{x}$ denoted by $g(\tilde{x}) \in \mathbb{R}^p$, base level of estimator $n_0 \geq 0$ and rate parameter $1 < \gamma < 2$. \\
		\STATE \textbf{Output:} $W \in \mathbb{R}^p$, a variance reduced unbiased estimator of the gradient of $ \mathbb{E}_vf(\mathbb{E}_w g(x, w), v)$ at point $x$.
		\STATE Sample $N$ from geometric distribution with success rate $1-p$ where $p=0.5^{\gamma}$ and let $\tilde{p}_N=(1-p)\cdot p^N$.
		\STATE Sample I.I.D $\{w_i\}_{1 \leq i \leq 2^{N+n_0+1}}$ follow the distribution of $w$ and obtain $\{S_i(x),T_i(x)\}_{1\leq i\leq 2^{N+n_0+1}}$.
		\STATE  Set $Y_1(x) = [\bar{S}_{2^{N+n_0+1}}(x)]^\intercal \cdot \nabla f_{v_1}(\bar{T}_{2^{N+n_0+1}}(x))$. $\quad$ Set $Y_1(\tilde{x}) = [\bar{S}_{2^{N+n_0+1}}(\tilde{x})]^\intercal \cdot \nabla f_{v_1}(\bar{T}_{2^{N+n_0+1}}(\tilde{x}))$.
		\STATE Set $Y_2(x) = [\bar{S}_{2^{N+n_0}}(x)]^\intercal \cdot \nabla f_{v_1}(\bar{T}_{2^{N+n_0}}(x))$. $\quad\qquad$ Set $Y_2(\tilde{x}) = [\bar{S}_{2^{N+n_0}}(\tilde{x})]^\intercal \cdot \nabla f_{v_1}(\bar{T}_{2^{N+n_0}}(\tilde{x}))$.
		\STATE Set $Y_3(x) = [\tilde{S}_{2^{N+n_0}}(x)]^\intercal \cdot \nabla f_{v_1}(\tilde{T}_{2^{N+n_0}}(x))$. $\quad\qquad$ Set $Y_3(\tilde{x}) = [\tilde{S}_{2^{N+n_0}}(\tilde{x}) \cdot]^\intercal \nabla f_{v_1}(\tilde{T}_{2^{N+n_0}}(\tilde{x}))$.
		\STATE Set $Y_4(x) = [\bar{S}_{2^{n_0}}(x)]^\intercal \cdot \nabla f_{v_1}(\bar{T}_{2^{n_0}}(x))$. $\qquad\quad\qquad$ Set $Y_4(\tilde{x}) = [\bar{S}_{2^{n_0}}(\tilde{x})]^\intercal \cdot \nabla f_{v_1}(\bar{T}_{2^{n_0}}(\tilde{x}))$.
		\STATE Set $W(x,v_1)=\frac{Y_1(x) -0.5 \cdot \{Y_2(x)+Y_3(x) \}}{\tilde{p}_N} + Y_4(x)$. $\quad$ Set $W(\tilde{x},v_1)=\frac{ Y_1(\tilde{x})-0.5 \cdot \{ Y_2(\tilde{x})+ Y_3(\tilde{x})\}}{\tilde{p}_N} + Y_4(\tilde{x}) $.
		\STATE Set $W=W(x,v_1)-W(\tilde{x},v_1) + g(\tilde{x})$.
		\STATE \textbf{Output:} W
	\end{algorithmic}
\end{algorithm}

In the above algorithm, the reference gradient $g(\tilde{x})$ can either be the full gradient at $\nabla F(\tilde{x})$ or some estimate of the full gradient $\nabla F(\tilde{x})$. Specifically, when it is efficient  to compute full gradients of the objective function for problem \eqref{prob:composition_stoc}, we propose to use the Variance Reduced Simulated Gradient Descent method for solving this problem.
\begin{algorithm}[H]
	\caption{Simulated Variance Reduced Gradient Descent(Simulated SVRG)}
	\label{svrggogo}
	\begin{algorithmic}
		\STATE\textbf{Inputs:} Number of epochs $T$, number of steps in each epoch $M$, step size $\lambda$ and initial point $\tilde{x}_0$.
		\FOR {$s=1,2,...T$}
		\STATE $\tilde{h}=\nabla F(\tilde{x}_{s-1})$
		\STATE $x_0=\tilde{x}$
		\FOR {$t=1,2,...M$}
		\STATE  Sample $v_t$ from the distribution of $v$. 
		\STATE Set $\nu_t =  \text{SimulatedGradient}(x_{t-1}, \tilde{x}_{s-1}, \tilde{h},v_t)$.
		\STATE Update ${x}_{t}={x}_{t-1}-\lambda \nu_t$.
		\ENDFOR
		\STATE \textbf{option I} Output $\tilde{x}_s=x_{M}$
		\STATE \textbf{option II} Output $\tilde{x}_s=x_{t}$ for randomly chosen $t \in \{0,...,M-1\}$
		\ENDFOR
		\end{algorithmic}
\end{algorithm}
However, when the full gradients $\nabla F(\tilde{x})$ of the objective function \eqref{prob:composition_stoc} can not be computed efficiently, we estimate the full gradient $\nabla F(\tilde{x})$ by sampling unbiased gradient within a batch of the index and taking avergae. We summarize the detail into the following Stochastically Controlled Simulated Gradient method. 
\begin{algorithm}[H]
	\caption{Stochastically Controlled Simulated Gradient Descent(Simulated SCSG)}
	\label{svrgalgo}
	\begin{algorithmic}
		\STATE\textbf{Inputs:} Number of epochs $T$, number of steps in each epoch $M$, batch size $B$, sample size $K$, step size $\lambda$, initial point $\tilde{x}_0$.
		\FOR {$s=1,2,...T$}
		\STATE $\tilde{x}=\tilde{x}_{s-1}$
		\STATE Uniformly sample a batch $\mathcal{I}_s \subset \Omega_v$ according to the distribution of $v$ with $|{\mathcal{I}_s}|=B$
		\FOR {$k=1,2,...,K$}
		\STATE Generate $h_k(\tilde{x})=\frac{1}{B}\sum_{v_i\in \mathcal{I}_s}$UnbiasedGradient$(\tilde{x},v_i)$
		\ENDFOR
		\STATE Set $\tilde{h}(\tilde{x})=\frac{1}{K}\sum_{i=1}^{K} h_i(\tilde{x})$
		\FOR {$t=1,2,...M$}
		\STATE  Sample $v_t$ from the distribution of $v$.
		\STATE Set $\nu_t = \text{SimulatedGradient}(x_{t-1}, \tilde{x}_{s-1}, \tilde{h}(\tilde{x}),v_t)$.
		\STATE Update ${x}_{t}={x}_{t-1}-\lambda v_t$.
		\ENDFOR
		\STATE \textbf{option I} Output $\tilde{x}_s=x_{M}$
		\STATE \textbf{option II} Output $\tilde{x}_s=x_{t}$ for randomly chosen $t \in \{0,...,M-1\}$
		\ENDFOR
		
	\end{algorithmic}
\end{algorithm}
We will prove the convergence properties of Algorithm 4 and 5 in section 4.
\section{Examples}
We present some important examples of stochastic optimization problem.
\subsection{Conditional Random Fields (CRF)}
Conditional random fields \cite{lafferty2001conditional} is a popular probabilistic model used for structural prediction. It has been used in a number of  natural language processing problems including  part-of-speech tagging \cite{lafferty2001conditional}, noun-phrase chunking \cite{sutton2007dynamic, sha2003shallow} named identity recognition \cite{mccallum2003early} and image segmentation task in computer vision \cite{nowozin2011structured}. For example, Given an observation $x \in \mathcal{X}$, the conditional probability of a structured outcome $y \in \mathcal{Y}$ is given by
\begin{align}
	p(y \, \vert \, x ; \theta) = \frac{\exp\{\theta^\top F(x, y)\}}{\sum_{y' \in \mathcal{Y}} \exp\{\theta^\top F(x, y')\}},
\end{align}
where $\theta \in \mathbb{R}^p$ is the parameter to be estimated and $F(x, y) \in \mathbb{R}^p$ is pre-specified feature functions depending on the underlying structure of $\mathcal{Y}$. Base on a set of training data $\{(x_i, y_i), i = 1, \ldots, n\}$, the parameter $\theta$ can be estimated by maximize the log likelihood function
\begin{align}
	\max_{\theta \in \mathbb{R}^p} \frac{1}{n} \sum_{i = 1}^n \log p(y_i \, \vert\, x_i, \theta). \label{prob:crf}
\end{align}
The key difficulty of computing the objective function value or its gradient lies in the exponential cardinality of $\mathcal{Y}$. When the underlying structure of $\mathcal{Y}$ is a linear chain or a tree, both objective function values and its gradient can be efficiently computed by dynamic programming method (the Viterbi algorithm\cite{1450960}). In this case, a number of  methods could be used to solve \eqref{prob:crf}, for example, deterministic methods such as the iterative scaling algorithm in \cite{lafferty2001conditional} , L-BFGS \cite{sha2003shallow}, stochastic methods such as SGD in \cite{vishwanathan2006accelerated} and SAG in \cite{schmidt2015non}. However, the computational issue of CRF has not been fully addressed when the underlying structure is more complex. In our setting, we can formulate   \eqref{prob:crf} as a composition optimization problem as in \eqref{prob:composition_stoc} by noticing that \eqref{prob:crf} is equivalent to
\begin{align*}
	\min_{\theta} \frac{1}{n}\sum_{i=1}^n \big( \log\big[  \sum_{y' \in \mathcal{Y}} \exp\{\theta^\top F(x_i, y')\}\big] - \theta^\top F(x_i, y_i) \big)\label{prob:crf1}
\end{align*}
whose gradient can be written as 
\begin{align*}
	\frac{1}{n}\sum_{i=1}^n \frac{\sum_{y' \in \mathcal{Y}}\exp\{\theta^\top F(x_i, y')\}F(x_i, y')}{\sum_{y' \in \mathcal{Y}} \exp\{\theta^\top F(x_i, y')\}} - F(x_i, y_i).
\end{align*}
Note that this problem is equivalent to 
\begin{align*}
		\min_{\theta} \frac{1}{n}\sum_{i=1}^n \big( \log\big[ \frac{1}{|\mathcal{Y}|} \sum_{y' \in \mathcal{Y}} \exp\{\theta^\top F(x_i, y')\}\big] - \theta^\top F(x_i, y_i) + \log |\mathcal{Y}| \big).
\end{align*}
Therefore we can view it as a function composition and apply our optimization algorithms to solve this problem.

\subsection{Cox's partial likelihood}
Cox's partial likelihood\cite{david1972regression, cox1975partial} is a widely used model in survival analysis for censored data. The model assumes 
\begin{align*}
	\lambda(t \vert X) = \lambda_0(t)\exp(\beta' X),
\end{align*}
where $\lambda(t \vert X)$ is the hazard function for an individual with covariates $X \in \mathbb{R}^p$ and coefficient $\beta \in \mathbb{R}^p$; and $\lambda_0(t)$ is the baseline hazard function. In the model,  let $(X_i, Y_i, \Delta_i)_{1\leq i \leq n}$ be i.i.d. observations where $X_i \in \mathbb{R}^p$ is the covariates and let $Y_i = \min(T_i, C_i)$ , $\Delta_i = \mathbb{I}\{Y_i = T_i\}$ where $T_i$ is the true life time and $C_i$ is the censoring time independent of $T_i$. Also, for a particular observation $i$, its risk set is defined to be the index set $\{j: {Y_j \geq Y_i}\}$. The goal is to maximize the partial likelihood function which can be written as the following composition optimization problem as in \eqref{prob:composition_stoc}:
\begin{align}
	\min_{\beta \in \mathbb{R}^p} \frac{1}{n}\sum_{i=1}^n\Delta_i[-X_i^\top \beta + \log\{\sum_{j = 1}^n \mathbb{I}(Y_j \geq Y_i)\exp(X_j^\top \beta)\}],
\end{align}
  and the gradient of this objective function is
\begin{align}
	\frac{1}{n}\sum_{i=1}^n \Delta_i[-X_i + \frac{\sum_{j=1}^n \mathbb{I}(Y_j\geq Y_i )\exp(X_j^\top \beta) X_j}{\sum_{j=1}^n \mathbb{I}(Y_j \geq Y_i) \exp(X_i^\top \beta)}].
\end{align}
Note that this problem is equivalent to
\begin{align*}
	\min_{\beta \in \mathbb{R}^p} \frac{1}{n}\sum_{i=1}^n\Delta_i[-X_i^\top \beta + \log\{\frac{1}{n}\sum_{j = 1}^n \mathbb{I}(Y_j \geq Y_i)\exp(X_j^\top \beta)\}].
\end{align*}
Now we can view this problem as a composition of functions and apply the proposed algorithm to solve it.

\subsection{Solving expectation-Maximization (EM) subproblem without posterior sampling}
An Expectation-Maximization (EM) algorithm \cite{dempster1977maximum} is an iterative procedure to obtain an MLE of a statistical model with the presence of latent variables (or random effects). Given the observed data ${x}$, latent data $z$, the parameters to be estimated $\theta \in \mathbb{R}^p$ and the likelihood function $L(\theta; x, z) = p(x, z \ \vert\ \theta)$, the EM algorithm iteratively performs the following two steps
\begin{itemize}
	\item{E-step} Update $Q(\theta \vert \theta^{(t)}) = \int \log L(\theta ; x, z) p(z \vert x, \theta^{(t)}) dz$
	\item{M-step} $\text{maximize}_{\theta \in \mathbb{R}^p} Q(\theta \vert \theta^{(t)})$.
\end{itemize}
When the latent variable $z$ is high dimensional, due to the difficulty of numerical integration in E-step, the two steps are combined into a stochastic optimization problem:
\begin{align}
	\min_{\theta \in \mathbb{R}^p} - \int \log L(\theta ; x, z) p(z \vert x, \theta^{(t)}) dz.
\end{align}
This problem can be solved by sampling from the posterior distribution and applying stochastic gradient descent algorithm. However, the Markov chain Monte Carlo (MCMC) algorithms used for  posterior sampling can be slow and inaccurate in high dimensional cases.Therefore, we rewrite the objective function as
\begin{align}
	-\int \log L(\theta; x, z) p(z \vert x, \theta^{(t)})dz &= -\int \log L(\theta; x, z) \frac{p(x \vert z, \theta^{(t)})}{p(x \vert \theta^{(t)})}p(z)dz \nonumber \\
	&= -\int \log L(\theta; x, z) \frac{p(x \vert z, \theta^{(t)})}{\int p(x \vert z, \theta^{(t)}) p(z)dz}p(z)dz,
\end{align}
and treat it as minimizing function compositions using the proposed algorithms.
\section{Theory}
In this section we present the analysis of our algorithms applied for problem \eqref{prob:composition_stoc}, the case where one is sloving $\min\limits_{x \in \mathbb{R}^p} F(x) \triangleq \mathbb{E}_v f\{\mathbb{E}_w g(x, w), v\} $. The case for \eqref{prob:composition_finite} and \eqref{prob:composition_mix} can be analyzed similarly. 
\subsubsection{Definitions and Assumptions}
\begin{assumption}
Given the initial point $\tilde{x}_0 \in \mathbb{R}^p$, there exist a compact set  $\mathcal{D}$ such that then the sequence of iterates $\{x_k\}_{k\geq 0}$ produced by the algorithms is contained in $\mathcal{D}$.
\end{assumption}
\begin{assumption}
	Inside the compact set $\mathcal{D}$, each $f_v(\cdot)$ in the objective function of \eqref{prob:composition_stoc} is $\mu$-strongly convex with $L$-Lipschitz continuous gradients.
\end{assumption}
\begin{assumption}\label{lp}
Inside the compact set $\mathcal{D}$, each $f_v(\cdot) $ is twice continously differentiable and its second derivatives have $L$-Lipschitz continous gradient and each $g_w(\cdot)$ is twice continously differentiable. 
\end{assumption}
\textbf{Remark}: Assumption 1 is reasonable for deterministic SVRG and SCSG algorithms. In the Simulated SVRG and SCGS algorithms where we use simulated gradients, we can still justify it under small adjustments. For example, if we switch the Simulated SVRG and SCGS to the deterministic ones whenever the output $\tilde{x}_s$ of the algorithm lies outside some compact set $\mathcal{D}_0$, then the convergence result for the algorithms will not be affected while we may find a appropriate $\mathcal{D}_0 \subset \mathcal{D}$ where assumption 1 holds. In practice, by making $\mathcal{D}_0$ large enough, the adjustment will not be necessary.

	  \begin{definition}\label{ld}
	  	We define the support of distribution $v$ and $w$ to be $\Omega_v$ and $\Omega_w$. Let $\mathcal{G} = \{y \in \mathbb{R}^d \, | \, y = g_w(x), x \in \mathcal{D},w\in \Omega_w\}$ $\mathcal{H} = \{y \in \mathbb{R}^{d \times p} \, | \, y = \nabla g_w(x), x \in \mathcal{D},w\in \Omega_w\}$ and $\mathcal{J} = \{z \in \mathbb{R}^{d \times p\times p} \, | \, z = \nabla^2 g_w(x), x \in \mathcal{D},w\in \Omega_w\}$. We define $l_{f}=\sup\limits_{y\in \mathcal{G}}\sup\limits_{v \in \Omega_v}\sup\limits_{0\leq i \leq 2} |f_v^{(i)}(y)|$ and $l_{g}=\sup\limits_{x\in \mathcal{D}}\sup\limits_{w \in \Omega_w}\sup\limits_{0\leq j \leq 2} |g_w^{(j)}(x)|$ where we write the upper index $f^{(i)}$ and $g^{(j)}$ to denote the order of the derivative when they are actually partial derivatives and the norm $|\cdot|$ is taken to be the maxinum among partial derivatives.
	  	
	   Finally, we set $l_{\mathcal{D}}=max\{l_f,l_g,L,1\}$ so that the norm of any component of the partial derivative functions $f_v^{(i)}(y),g_{w}^{(j)}(x)$ is  bounded by $l_{\mathcal{D}}$ with Lipschitz continous gradient $l_{\mathcal{D}}$ for any $x\in \mathcal{D}, y\in \mathcal{G},v \in \Omega_v, w\in \Omega_w$ and $0 \leq i \leq 2, 0\leq j \leq 1$. As a consequence of Assumption 1 and Assumption 3, we have $l_{\mathcal{D}} < \infty$.
	\end{definition}
Before we proceed to the proofs, we introduce some techinal lemmas.
	\begin{Lemma}\label{secondorder}
		Let $f:\mathbb{R}^d \rightarrow \mathbb{R}$ be a continuously differentiable function with $L$-Lipschitz continuous gradients, then
		\begin{align*}
		|f(y) - f(x) - \langle \nabla f(x), y-x \rangle | \leq \frac{L}{2}\|y -x \|_2^2.
		\end{align*}
	\end{Lemma}
	\begin{proof}
		Let $g(x)=\frac{L}{2}x^{T}x-f(x)$. Since $f(\cdot)$ has $L$-Lipschitz continous gradient where $\|\nabla f(y)-\nabla f(x)\|_2 \leq L\|y-x\|_2$, we have $(\nabla f(y)-\nabla f(x))^T (y-x)\leq L\|y-x\|_2^2$ by Cauchy-Schwartz. This implies
		\begin{align*}
		(\nabla g(y)-\nabla g(x))^T (y-x)=L\|y-x\|_2^2 - (\nabla f(y)-\nabla f(x))^T (y-x) \geq 0,
		\end{align*}
		which shows $g(\cdot)$ is convex. The convexity of $g(\cdot)$ implies
		\begin{equation*}
		0 \leq g(y)-g(x)-\nabla g(x)^T(y-x)=\frac{L}{2}\|y-x\|_2^2 -(f(y)-f(x)-\langle \nabla f(x), y-x\rangle).
		\end{equation*}
		So we have $f(y)-f(x)-\langle \nabla f(x), y-x\rangle \leq \frac{L}{2}\|y-x\|_2^2$. Since $-f(\cdot)$ also has $L$-Lipschitz continous gradient, we can substitute $f(\cdot)$ with $-f(\cdot)$ from the above equation and deduce that $|f(y)-f(x)-\langle \nabla f(x), y-x\rangle| \leq \frac{L}{2}\|y-x\|_2^2$.
		\end{proof}

	\begin{Lemma}\label{lipschitz}
	Let $\{f_i(\cdot)\}_{1 \leq i \leq n}:\mathbb{R}^d \rightarrow \mathbb{R}$ be bounded Lipschitz function. Then $\sum_{i=1}^n f_i(\cdot)$, $\prod_{i=1}^n f_i(\cdot)$ and $[f_1(\cdot),...,f_n(\cdot)]^\intercal$ are also Lipschitz functions. 
	\end{Lemma}
	\begin{proof}
		Suppose the $|f_i(\cdot)| \leq 1$ and $|f_i(x_1)-f_i(x_2)| \leq L \|x_1-x_2\|_2$ for any $x_1,x_2 \in \mathbb{R}^d$ and $1 \leq i \leq n$, then
	\begin{align*}
	|(\sum_{i=1}^n f_i)(x_1)-(\sum_{i=1}^n f_i)(x_2)| \leq nL\|x_1-x_2\|_2.
	\end{align*}
	On the other hand,
	\begin{align*}
	|(f_1\cdot f_2)(x_1)-(f_1\cdot f_2)(x_2)| &\leq |f_1(x_1)f_2(x_1)-f_1(x_1)f_2(x_2)+f_1(x_1)f_2(x_2)-f_1(x_2)f_2(x_2)| \nonumber\\
	& \leq |f_1(x_1)||f_2(x_1)-f_2(x_2)|+|f_2(x_2)||f_1(x_1)-f_1(x_2)|\nonumber\\
	& \leq (L+L)\|x_1-x_2\|_2.
	\end{align*}
	Since $|f_1\cdot f_2|\leq 1$, it follows from induction that $\prod_{i=1}^n f_i(\cdot)$ is Lipschitz continous with constant $nL$.For the general case where $|f_i(\cdot)| \leq M$, apply the lemma to each $\frac{f_i(\cdot)}{M}$, the Lipschitz constant of $\prod_{i=1}^n f_i(\cdot)$ becomes $nM^nL$.
	
	Finally, considering the function $[f_1(x),...,f_n(x)]^\intercal:\mathbb{R}^{d}\rightarrow \mathbb{R}^n$
	\begin{align*}
	\|[f_1(x_1),...,f_n(x_1)]^\intercal-[f_1(x_2),...,f_n(x_2)]^\intercal\|_2 =\sqrt{\sum_{j=1}^n (f_j(x_1)-f_j(x_2))^2} \leq \sqrt{n} L \|x_1-x_2\|_2 
	\end{align*}
	\end{proof}
\begin{Lemma}
Given a sequence of real number $a_i$, $1 \leq i \leq N$ and a positive integer $N$, we have
\begin{equation}\label{elementary}
|\sum_{i=1}^N a_i|^p \leq N^{p-1} \sum_{i=1}^N |a_i|^p
\end{equation}
\end{Lemma}
	
	\subsection{Properties of the Unbiased  Gradient Simulation Algorithm}
	In the following section, we present some properties of the output $W(x,v_1)$ from Algorithm 1. We first prove the unbiasedness of $W(x,v_1)$.
	\begin{Proposition}
		For any $x \in \mathcal{D}$, sample $v_1 \sim v$, then $W(x,v_1)$ is an unbiased estimate of $ \nabla \mathbb{E}_v f_v\{\mathbb{E}_w g_w(x)\}$.
	\end{Proposition}
	\begin{proof}
		Fix $v_1$ and $x \in \mathcal{D}$, we will show that the output $W(x,v_1)$ is an unbiased estimate of $ f_{v_1}\{\mathbb{E}_w g_w(x)\}$. According to Algorithm 1, we have,
		\begin{align*}
			 \mathbb{E}W(x,v_1)
			 =&\sum_{n=0}^{\infty} \mathbb{E}[W(x,v_1)|N=n]\cdot \mathbb{P}(N=n) = \sum_{n=0}^{\infty} \frac{\mathbb{E}[{Y_1 - 0.5(Y_2 + Y_3)}|N=n]}{\tilde{p}_n}\cdot \tilde{p}_n +\mathbb{E}Y_4  \nonumber \\
			 =&\sum_{n = 0}^\infty \mathbb{E}  \bigg([\bar{S}_{2^{n+n_0+1}}(x)]^\intercal \cdot \nabla f_{v_1}(\bar{T}_{2^{n+n_0+1}}(x))-0.5\Big([\bar{S}_{2^{n+n_0}}(x)]^\intercal \cdot \nabla f_{v_1}(\bar{T}_{2^{n+n_0}}(x))+[\tilde{S}_{2^{n+n_0}}(x)]^\intercal \cdot \nabla f_{v_1}(\tilde{T}_{2^{n+n_0}}(x))\Big)\bigg) \\
			&\quad+ \mathbb{E} \big[[\bar{S}_{2^{n_0}}(x)]^\intercal \cdot \nabla f_{v_1}(\bar{T}_{2^{n_0}}(x))\big] \nonumber\\
			=&\Big(\sum_{n=0}^{\infty} \mathbb{E}  [[\bar{S}_{2^{n+n_0+1}}(x)]^\intercal \cdot \nabla f_{v_1}(\bar{T}_{2^{n+n_0+1}}(x))]-\mathbb{E}[[\bar{S}_{2^{n+n_0}}(x)]^\intercal \cdot \nabla f_{v_1}(\bar{T}_{2^{n+n_0}}(x))] \Big) + \mathbb{E}[[\bar{S}_{2^{n_0}}(x)]^\intercal \cdot \nabla f_{v_1}(\bar{T}_{2^{n_0}}(x))] \nonumber\\
			=& \lim_{n \rightarrow \infty} \mathbb{E}  \big([\bar{S}_{2^{n+n_0+1}}(x)]^\intercal \cdot \nabla f_{v_1}(\bar{T}_{2^{n+n_0+1}}(x))
			= \mathbb{E}[S(x)]^\intercal \cdot \nabla f_{v_1}(\mathbb{E}[T(x)]) 
			= [\mathbb{E}_w \nabla g_w(x)]^\intercal \cdot \nabla f_{v_1}(\mathbb{E} g_w(x))\nonumber\\
			=& \nabla (f_{v_1}(\mathbb{E}_w g_w(x)))
		\end{align*}
	where the third inequality follows from the fact that $\mathbb{E}[[\bar{S}_{2^{n+n_0}}(x)]^\intercal \cdot \nabla f_{v_1}(\bar{T}_{2^{n+n_0}}(x))]=\mathbb{E}[[\tilde{S}_{2^{n+n_0}}(x)]^\intercal \cdot \nabla f_{v_1}(\tilde{T}_{2^{n+n_0}}(x))]$ for any $n$. The next line follows from the continuity of $\nabla f_{v_1}(\cdot)$, Assumptions 2-3 and the bounded convergence theorem.
	
	Finally, taking expectation w.r.t $v_1$, we can see that $\mathbb{E}_{v_1\sim v} W(x,v_1)=\mathbb{E}_v \nabla (f_v(\mathbb{E}_w g_w(x)))=  \nabla (\mathbb{E}_vf_v(\mathbb{E}_w g_w(x)))$.
	\end{proof}
We now proceed to show $W(x,v_1)$ has finite variance and finite expectation cost to generate.
	\begin{Proposition}\label{finitev}
		 For any $x \in \mathcal{D}$, sample $v_1 \sim v$, then  $W(x,v_1)$ has finite variance.
	\end{Proposition}
	\begin{proof}
	Fix  $v_1 \in \Omega_v$ and $x \in \mathcal{D}$, we will show that $\mathbb{E}\|W(x,v_1)\|_2^2 < \infty$.
	\begin{align}\label{variance}
		\mathbb{E}\|W(x,v_1)\|_2^2 =&\sum_{n=0}^{\infty} \mathbb{E}[\|W(x,v_1)\|_2^2|N=n]\cdot \mathbb{P}(N=n) \leq 2\mathbb{E}\|Y_4\|_2^2+2\frac{\mathbb{E}[\|Y_1 - 0.5(Y_2 + Y_3)\|_2^2|N=n]}{\tilde{p}^2_n} \cdot \tilde{p}_n 
	\end{align}
	where the inequality follows from equation \eqref{elementary}. To proceed with equation \eqref{variance}, notice  $2\mathbb{E}\|Y_4\|^2_2$ can be bounded by $2p \cdot d^2 \cdot l^4_{\mathcal{D}}$ since   $\|Y_4\|^2_2=\|[\bar{S}_{2^{n_0}}]^\intercal \cdot \nabla f(\bar{Y}_{2^{n_0}})\|_2^2 \leq p \cdot d^2 \cdot l^4_{\mathcal{D}} $ based on the definition of $l_{\mathcal{D}}$. 
	
	To bound the second term on the right hand side of \eqref{variance}, we define the following function:
for $x \in \mathcal{H} \subseteq \mathbb{R}^{d \times p }$ and  $y \in \mathcal{G} \subseteq \mathbb{R}^d$, define $G:\mathcal{H} \times \mathcal{G} \rightarrow \mathbb{R}^p$ by $G(x, y) \triangleq x^\intercal \cdot \nabla f_v(y)$. It is straightforward to compute each component of the gradient $\nabla G(x, y) \in \mathbb{R}^p \times (\mathbb{R}^{d \times p} \times \mathbb{R}^d)$ as: 
\begin{align*}
& \frac{\partial[G]_i}{\partial [x]_{kj}}(x,y)= \delta_{ij}\cdot [\nabla f_v(y)]_k=\delta_{ij}\cdot \frac{\partial f_v}{\partial [y]_k}(y) \quad \text{and} \quad \frac{\partial[G]_i}{\partial [y]_{h}}(x,y)=\sum_{k=1}^d [x]_{ki}\cdot \frac{\partial [\nabla f_v]_k}{\partial [y]_h}=\sum_{k=1}^d [x]_{ki}\cdot \frac{\partial^2  f_v}{\partial [y]_k\partial [y]_h}(y)
\end{align*}
where $1\leq i,j\leq p$,$1 \leq k,h \leq d$ and $\delta_{ij}$ is the Kronecker delta ($\delta_{ij}=1$ when $i=j$; $\delta_{ij}=0 $ otherwise). It follows from Assumptions 1, Assumption 3 and Lemma \eqref{lipschitz} that each component of the $\nabla G(x,y)$ is Lipschitz continous with constant $d\cdot l^2_{\mathcal{D}}$ and $\nabla [G(x,y)]_i$ is Lipschitz continous with constant $\sqrt{d\cdot p+d}\cdot d \cdot l^2_{\mathcal{D}}$. Thus if we let $R(x,x_0,y,y_0)\triangleq G(x,y)-G(x_0,y_0)-\nabla G(x_0,y_0) \cdot vec \binom{x-x_0}{y-y_0}$, using Lemma \eqref{secondorder}, we have:
\begin{align}\label{reminder}
\|R(x,x_0,y,y_0)\|_2=&\|G(x,y)-G(x_0,y_0)-\nabla G(x_0,y_0) \cdot  \binom{x-x_0}{y-y_0}\|_2 \nonumber\\
 \leq& \sum_{i=1}^p \|[G(x,y)]_i-[G(x_0,y_0)]_i-\nabla[G(x_0,y_0)]_i\cdot  \binom{x-x_0}{y-y_0}\|_2 \nonumber \\
 \leq &\sum_{i=1}^p \frac{\sqrt{dp+d}\cdot d\cdot l^2_{\mathcal{D}}}{2}(\|x-x_0\|^2_F+\|y-y_0\|_2^2) \nonumber \\
= & \frac{p\cdot \sqrt{dp+d}\cdot d\cdot l^2_{\mathcal{D}}}{2} (\sum_{k=1}^d\sum_{j=1}^p ([x]_{kj}-[x_0]_{kj})^2+\sum_{h=1}^d ([y]_h-[y_0]_h)^2),
\end{align}
for any $x,x_0\in \mathcal{H}$ and $y,y_0 \in \mathcal{G}$.
Now we can bound the second term in \eqref{variance}:
\begin{align}\label{fourth}
&\frac{\mathbb{E}[\|Y_1 - 0.5(Y_2 + Y_3)\|_2^2|N=n]}{\tilde{p}^2_n} \cdot \tilde{p}_n \nonumber\\
=&\sum_{n=0}^{\infty}\frac{1}{\tilde{p}_n}\mathbb{E}\|[\bar{S}_{2^{n+n_0+1}}(x)]^\intercal \cdot \nabla f_{v_1}(\bar{T}_{2^{n+n_0+1}}(x))-0.5\Big([\bar{S}_{2^{n+n_0}}(x)]^\intercal \cdot \nabla f_{v_1}(\bar{T}_{2^{n+n_0}}(x))+[\tilde{S}_{2^{n+n_0}}(x)]^\intercal \cdot \nabla f_{v_1}(\tilde{T}_{2^{n+n_0}}(x))\Big)\|_2^2 \nonumber\\
=& \sum_{n=0}^{\infty}\frac{1}{\tilde{p}_n}\mathbb{E}\|G(\bar{S}_{2^{n+n_0+1}}(x), \bar{T}_{2^{n+n_0+1}}(x))-0.5\Big(G(\bar{S}_{2^{n+n_0}}(x),\bar{T}_{2^{n+n_0}}(x))+G(\tilde{S}_{2^{n+n_0}}(x),\tilde{T}_{2^{n+n_0}}(x))\Big)\|_2^2 \nonumber\\
=& \sum_{n=0}^{\infty}\frac{1}{\tilde{p}_n}\mathbb{E}\|G(\mathbb{E}S(x), \mathbb{E}T(x))+\nabla G(\mathbb{E}S(x), \mathbb{E}T(x))\cdot  \binom{\bar{S}_{2^{n+n_0+1}}(x)-\mathbb{E}S(x)}{\bar{T}_{2^{n+n_0+1}}(x)-\mathbb{E}T(x)}+ R(\bar{S}_{2^{n+n_0+1}}(x),\mathbb{E}S(x),\bar{T}_{2^{n+n_0+1}}(x),\mathbb{E}T(x)) \nonumber \\
& - G(\mathbb{E}S(x), \mathbb{E}T(x))-\nabla G(\mathbb{E}S(x), \mathbb{E}T(x))\cdot  \binom{\frac{\bar{S}_{2^{n+n_0}}(x)+\widetilde{S}_{2^{n+n_0}}(x)}{2}-\mathbb{E}S(x)}{\frac{\bar{T}_{2^{n+n_0}}(x)+\widetilde{T}_{2^{n+n_0}}(x)}{2}-\mathbb{E}T(x)}\nonumber\\
 & -\frac{1}{2}R(\bar{S}_{2^{n+n_0}}(x),\mathbb{E}S(x),\bar{T}_{2^{n+n_0}}(x),\mathbb{E}T(x)) -\frac{1}{2}R(\bar{S}_{2^{n+n_0}}(x),\mathbb{E}S(x),\bar{T}_{2^{n+n_0}}(x),\mathbb{E}T(x))\|_2^2 \nonumber \\
 =& \sum_{n=0}^{\infty}\frac{1}{\tilde{p}_n}\mathbb{E}\|R(\bar{S}_{2^{n+n_0+1}}(x),\mathbb{E}S(x),\bar{T}_{2^{n+n_0+1}}(x),\mathbb{E}T(x))\nonumber \\
 &-\frac{1}{2}R(\bar{S}_{2^{n+n_0}}(x),\mathbb{E}S(x),\bar{T}_{2^{n+n_0}}(x),\mathbb{E}T(x)) -\frac{1}{2}R(\bar{S}_{2^{n+n_0}}(x),\mathbb{E}S(x),\bar{T}_{2^{n+n_0}}(x),\mathbb{E}T(x))\|_2^2 \nonumber \\
 \leq & \sum_{n=0}^{\infty}\frac{3}{4\tilde{p}_n} p^3(dp+d)d^4l_{\mathcal{D}}^4  \bigg(
 \sum_{h=1}^d\mathbb{E}([\bar{T}_{2^{n+n_0+1}}(x)]_{h}-[\mathbb{E}T(x)]_{h})^4+\frac{1}{4}\mathbb{E}([\bar{T}_{2^{n+n_0}}(x)]_{h}-[\mathbb{E}T(x)]_{h})^4+\frac{1}{4}\mathbb{E}([\tilde{T}_{2^{n+n_0}}(x)]_{h}-[\mathbb{E}T(x)]_{h})^4 \nonumber\\
 &+\sum_{k=1}^d\sum_{j=1}^p \mathbb{E}([\bar{S}_{2^{n+n_0+1}}(x)]_{kj}-[\mathbb{E}S(x)]_{kj})^4+ \frac{1}{4}\mathbb{E}([\bar{S}_{2^{n+n_0}}(x)]_{kj}-[\mathbb{E}S(x)]_{kj})^4 +\frac{1}{4}\mathbb{E}([\tilde{S}_{2^{n+n_0}}(x)]_{kj}-[\mathbb{E}S(x)]_{kj})^4 \bigg),
\end{align}
where the last inequality follows from equation \eqref{elementary} and  \eqref{reminder}. The equality above it  follows from equation \eqref{SandT}. 
To proceed with equation \eqref{fourth}, we notice that  $\{[S_i(x)]_{kj}$,$[T_i(x)]_h\}_{i \geq 1}$ are I.I.D samples with finite fourth central moment ($\max\{\mathbb{E}([S(x)]_{kj}-[\mathbb{E}S(x)]_{kj})^4,\mathbb{E}([T(x)]_{h}-[\mathbb{E}T(x)]_{h})^4\} \leq l^4_{\mathcal{D}}$) for any $1 \leq j \leq p, 1 \leq k,h \leq d$. Thus we can use Cauchy-Schwartz inequality and equation \eqref{elementary} to derive:
\begin{align}\label{fourthmomentbound}
&\mathbb{E}([\bar{S}_n(x)]_{kj}-[\mathbb{E}S(x)]_{kj})^4 \nonumber\\
=& \frac{\sum_{i=1}^n \mathbb{E}([S_i(x)]_{kj}-[\mathbb{E}S(x)]_{kj})^4}{n^4}+\frac{6\sum_{i=1}^n\sum_{j=i+1}^n \mathbb{E}([S_i(x)]_{kj}-[\mathbb{E}S(x)]_{kj})^2\mathbb{E}([S_j(x)]_{kj}-[\mathbb{E}S(x)]_{kj})^2}{n^4} \nonumber \\
\leq& \frac{\sum_{i=1}^n \mathbb{E}([S_i(x)]_{kj}-[\mathbb{E}S(x)]_{kj})^4}{n^4}+\frac{6\sum_{i=1}^n\sum_{j=i+1}^n \sqrt{\mathbb{E}([S_i(x)]_{kj}-[\mathbb{E}S(x)]_{kj})^4}\cdot\sqrt{\mathbb{E}([S_j(x)]_{kj}-[\mathbb{E}S(x)]_{kj})^4}}{n^4} \nonumber\\
\leq& l^4_{\mathcal{D}} (\frac{1}{n^3}+6\cdot\frac{n^2-n}{2n^4}) \leq 3\cdot\frac{l^4_{\mathcal{D}}}{n^2}.
\end{align}
for any $1 \leq j \leq p, 1 \leq k \leq d$ and $n \geq 1$. The same result holds for $\mathbb{E}([\bar{T}_n(x)]_{h}-[\mathbb{E}T(x)]_{h})^4$ where $1 \leq h \leq d$.
Now using \eqref{fourthmomentbound}, we continue on \eqref{fourth} to get
\begin{align}
\frac{\mathbb{E}[\|Y_1 - 0.5(Y_2 + Y_3)\|_2^2|N=n]}{\tilde{p}^2_n} \cdot \tilde{p}_n \leq \frac{3}{4\tilde{p}_n} p^3d^4(dp+d)\cdot l^4_{\mathcal{D}}\cdot 3(d+dp)\cdot \bigg( 4^{-n_0} \frac{3l^4_{\mathcal{D}}}{4}\cdot 2^{-2n}\bigg)
.\end{align}
Define $C^{\prime}_{\mathcal{D}} \triangleq 2p  d^2  l^4_{\mathcal{D}}+ \frac{54}{16\cdot4^{n_0}}p^3 d^4 (dp+d)^2 l^8_{\mathcal{D}} \cdot (1-0.5^{\gamma})^{-1} \cdot (1-2^{\gamma-2})^{-1}$. Notice $0<C^{\prime}_{\mathcal{D}} < \infty$ because $\gamma < 2$. Now \eqref{variance} becomes
\begin{align}\label{Cvariance}
\mathbb{E}\|W(x,v_1)\|_2^2 \leq& 2p\cdot d^2\cdot l^4_{\mathcal{D}}+ 2\frac{27}{16\cdot4^{n_0}}p^3 d^4 (dp+d)^2 l^8_{\mathcal{D}} \cdot (\sum_{n=0}^{\infty}\frac{2^{-2n}}{\tilde{p}_n}) 
= C^{\prime}_{\mathcal{D}},
\end{align}
following from the fact  $\tilde{p}_n=(1-0.5^{\gamma})\cdot{0.5}^{\gamma  n}$ and the definition of $C^{\prime}_{\mathcal{D}}$. It is worth noting that the convergence of the series above relies on $y<2$.

Finally, considering the case where we sample $v_1 \sim v$, we have 
\begin{equation*}
Var(W(x,v_1)) \leq \mathbb{E}_{v_1\sim v}\|W(x,v_1)\|_2^2 = \mathbb{E}[\mathbb{E}\|W(x,v_1)\|_2^2 |v_1] \leq \mathbb{E}[C^{\prime}_{\mathcal{D}}|v_1]=C^{\prime}_{\mathcal{D}} < \infty.
\end{equation*}
\end{proof}

\begin{Proposition}
	For any $x \in \mathcal{D}$ and $v_1 \in \Omega_v$, the number of random variables one needs to generate (simulation cost) in order to construct $W(x,v_1)$ has finite expectation.
\end{Proposition}
\begin{proof}
Fix $v_1 \in \Omega_v$ and $x \in \mathcal{D}$, denote $cost_W$ to be the number of random variables one needs to generate in order to construct $W(x,v_1)$. In Algorithm 1, we generate one geometric random variable $N$ and $2^{N+n_0+1}$ number of $w_i \sim w$. Thus we have $cost_W=1+2^{N+n_0+1}$. Taking expectation w.r.t. $N$, we conclude
\begin{align*}
\mathbb{E}[cost_W]=\mathbb{E}[\mathbb{E}[cost_W|N]]=&\sum_{n=0}^{\infty}\mathbb{E}[cost_W|N=n]\cdot p(n) \nonumber\\
= &\sum_{n=0}^{\infty} (1+2^{n+n_0+1})\cdot (1-0.5^{\gamma})\cdot 0.5^{\gamma n} \nonumber\\
=& 1+ 2^{n_0+1}\cdot(1-0.5^{\gamma})\cdot(1-2^{1-\gamma})^{-1} < \infty,
\end{align*}
where the convergence of the series above relies on $\gamma>1$.
\end{proof}.

\subsection{Properties of the Variance Reduced Unbiased Gradient Simulation Algorithm}
In the following section, we present a property of the output $W=W(x,v_1)-W(\tilde{x},v_1)+g(\tilde{x})$ from Algorithm 3 that is important in the proof for the convergence rate of Algorithm 4 and 5. We want to show there exist $C_{\mathcal{D}}< \infty$ such that for any $v_1 \in \Omega_v$ and $x,\tilde{x} \in \mathcal{D}$, we have $\mathbb{E}\|W(x,v_1)-W(\tilde{x},v_1)\|_2^2 \leq C_{\mathcal{D}}\|x-\tilde{x}\|_2^2$. In order to do so, we fisrt introduce a couple of lemmas.
\begin{Lemma}[Azuma-Hoeffding]\label{azuma}

Let $X_1,X_2,...X_n$ be I.I.D random variables such that $|X_i| \leq L$ for all $1 \leq i \leq  n$. Then for any $t > 0$, we have:
\begin{equation}
\mathbb{P}(|\sum_{i=1}^n X_i-n\mathbb{E}[X]|>t) \leq 2 \cdot exp \Big(-\frac{t^2}{2nL^2}\Big)
\end{equation}
which implies
\begin{equation}
\mathbb{P}(|\bar{X}_n-\mathbb{E}[X]|>t) \leq 2 \cdot exp \Big(-\frac{nt^2}{2L^2}\Big)
\end{equation}
for any $t>0$.
\end{Lemma}

\begin{Lemma}\label{finiteclt}
	Define $diam(\mathcal{D})\triangleq \sup \{\| x-\tilde{x}\|_{\infty}, | x,\tilde{x}\in \mathcal{D}\}$ and $C_{\mathcal{D}}^{\prime\prime} =(2)^{4-2.5p}(l_{\mathcal{D}})^4(2\cdot diam(\mathcal{D}))^p+16p^2l^4_{\mathcal{D}}$. Then:
	\begin{equation}\label{s4}
	\mathbb{E}\bigg[\sup_{\zeta \in \mathcal{D}}\lvert [\bar{S}_n(\zeta)]_{kj}-[\mathbb{E}S(\zeta)]_{kj} \rvert ^4\bigg] \leq C_{\mathcal{D}}^{\prime\prime}\big( \frac{log (4n)}{n}\big)^2
	\end{equation}
 
	\begin{equation}\label{t4} \mathbb{E}\bigg[\sup_{\zeta \in \mathcal{D}}\lvert [\bar{T}_n(\zeta)]_{h}-[\mathbb{E}T(\zeta)]_{h} \rvert ^4\bigg] \leq C_{\mathcal{D}}^{\prime\prime}\big( \frac{log (4n)}{n}\big)^2
	\end{equation}
	\begin{equation}\label {z4}  \mathbb{E}\bigg[\sup_{\zeta \in \mathcal{D}}\lvert [\bar{Z}_n(\zeta)]_{kij}-[\mathbb{E}Z(\zeta)]_{kij} \rvert ^4\bigg] \leq C_{\mathcal{D}}^{\prime\prime}\big( \frac{log (4n)}{n}\big)^2
	\end{equation}
	for any $n \geq 1$, $1 \leq k,h \leq d$ and  $1 \leq i,j \leq p$.
\end{Lemma}

\begin{proof}
	Fix $1 \leq k \leq d$ and $1 \leq j \leq p$, we prove equation \eqref{s4}. It follows from Definition \ref{ld} that for any $w \in \Omega_w$,
	\begin{equation}\label{28}
	\sup\limits_{\substack{x,\tilde{x} \in \mathcal{D}, \\ 1 \leq j \leq p \\ 1 \leq k \leq d}} \{\lvert \frac{\partial [g_{w}]_k}{\partial[x]_j}(x)-\frac{\partial [g_{w}]_k}{\partial[x]_j}(\tilde{x})\rvert \} \leq l_{\mathcal{D}} \|{x-\tilde{x}}\|_2 \leq l_{\mathcal{D}}\cdot \sqrt{p}\|{x-\tilde{x}}\|_{\infty} .
	\end{equation} 
	It also follows from Definition \ref{ld} that $diam(\mathcal{D}) < \infty$. Consequently, we can find a set $\Gamma \subset \mathbb{R}^p$ with $ |\Gamma| \leq {\big(\frac{2\cdot diam(\mathcal{D})}{\epsilon}\big)^p}$ such that for any $\zeta \in \mathcal{D}$, there exists $ \gamma \in \Gamma$ with $\|\zeta-\gamma\|_{\infty} \leq \epsilon$. 
	 It then follows from \eqref{28} that: 
	 \begin{align*}
	 |[\bar{S}_{n}]_{kj}(\zeta)-[\mathbb{E}S(\zeta)]_{kj} |=&|\frac{1}{n}\sum_{i=1}^n\frac{\partial [g_{w_i}]_k}{\partial[x]_j}(\zeta)-\mathbb{E}_{w}\frac{\partial [g_{w}]_k}{\partial[x]_j}(\zeta)| \nonumber\\
	  \leq& |\frac{1}{n}\sum_{i=1}^n\frac{\partial [g_{w_i}]_k}{\partial[x]_j}(\gamma)-\mathbb{E}_{w}\frac{\partial [g_{w}]_k}{\partial[x]_j}(\gamma)| +2\epsilon\sqrt{p}l_{\mathcal{D}}=|[\bar{S}_{n}]_{kj}(\gamma)-[\mathbb{E}S(\gamma)]_{kj} |+2\epsilon\sqrt{p}l_{\mathcal{D}}.
	 \end{align*}  
	 Thus, fix $\delta>0$ and $0<\epsilon<\min\{diam(\mathcal{D}),\frac{\delta}{2\sqrt{p}l_{\mathcal{D}}}\}$, we can denote the elements in $\Gamma$ by $\{\gamma_1,...,\gamma_{\lvert\Gamma\rvert}\}$ and write:
	\begin{align}\label{ah1p}
	\mathbb{P}\bigg\{\sup_{\zeta \in \mathcal{D}}|[\bar{S}_{n}]_{kj}(\zeta)-[\mathbb{E}S(\zeta)]_{kj} | \geq \delta \bigg\} &\leq \mathbb{P}\bigg\{\max
	_{1 \leq i \leq \lvert \Gamma \rvert}|[\bar{S}_{n}]_{kj}(\gamma_i)-[\mathbb{E}S(\gamma_i)]_{kj} | +2\epsilon \sqrt{p} l_{\mathcal{D}} \geq \delta \bigg\} \nonumber \\
	& \leq \sum_{i=1}^{\lvert \Gamma \rvert} \mathbb{P}\bigg\{|[\bar{S}_{n}]_{kj}(\gamma_i)-[\mathbb{E}S(\gamma_i)]_{kj} | \geq \delta-2\epsilon \sqrt{p} l_{\mathcal{D}}\bigg\} \nonumber \\
	& \leq \sum_{i=1}^{\lvert \Gamma \rvert} 2\cdot exp \{ -\frac{n(\delta-2\epsilon \sqrt{p} l_{\mathcal{D}})^2}{2l^2_{\mathcal{D}}} \} \leq \big(\frac{2\cdot diam(\mathcal{D})}{\epsilon}\big)^p \cdot exp{\{-\frac{n(\delta-2\epsilon \sqrt{p} l_{\mathcal{D}})^2}{2l^2_{\mathcal{D}}} \}},
	\end{align}
	where the third line follows from Lemma \ref{azuma}, the Azuma-Hoeffding inequality. Now we prove \eqref{s4}:
	\begin{align}\label{cpp}
	\mathbb{E}\bigg[\sup_{\zeta \in \mathcal{D}}|[\bar{S}_{n}]_{kj}(\zeta)-[\mathbb{E}S(\zeta)]_{kj} | ^4\bigg] &\leq (2l_{\mathcal{D}})^4 \mathbb{P}\{\sup_{\zeta \in \mathcal{D}}|[\bar{S}_{n}]_{kj}(\zeta)-[\mathbb{E}S(\zeta)]_{kj} | \geq \delta \} + \delta^4 \mathbb{P} \{\sup_{\zeta \in \mathcal{D}}|[\bar{S}_{n}]_{kj}(\zeta)-[\mathbb{E}S(\zeta)]_{kj} | < \delta \} \nonumber \\
	& \leq (2l_{\mathcal{D}})^4\big(\frac{2\cdot diam(\mathcal{D})}{\epsilon}\big)^p\cdot exp{\{-\frac{n(\delta-2\epsilon \sqrt{p} l_{\mathcal{D}})^2}{2l^2_{\mathcal{D}}} \}} + \delta^4  \nonumber \\
	&=(2l_{\mathcal{D}})^4(2\cdot diam(\mathcal{D}))^p\cdot exp{\{-\frac{n(\delta-2\epsilon \sqrt{p} l_{\mathcal{D}})^2}{2l^2_{\mathcal{D}}} +p \cdot log(\frac{1}{\epsilon})\}} + \delta^4,
	\end{align}
	where the second line follows from \eqref{ah1p}. Now take $\delta=\frac{2\sqrt{p}l_{\mathcal{D}}\cdot\sqrt{12log(2n+2)}}{\sqrt{2n}}$ and $\epsilon=\frac{1}{\sqrt{2n}}$. Without loss of generality, we assume $diam(\mathcal{D})>1$ so that $0<\epsilon<\min\{diam(\mathcal{D}),\frac{\delta}{2\sqrt{p}l_{\mathcal{D}}}\}$. Then \eqref{cpp} becomes:
	\begin{align}
	&\mathbb{E}\bigg[\sup_{\zeta \in \mathcal{D}}|[\bar{S}_{n}]_{kj}(\zeta)-[\mathbb{E}S(\zeta)]_{kj} |^4\bigg] \nonumber\\
	\leq& (2l_{\mathcal{D}})^4(2\cdot diam(\mathcal{D}))^p\cdot exp{\{ -p(\sqrt{12log(2n+2)}-1)^2+ p\cdot log(\sqrt{2n}) \  \}} +4\cdot12^2p^2l^4_{\mathcal{D}}(\frac{log(2n+2)}{n})^2 \nonumber \\
	\leq& (2l_{\mathcal{D}})^4(2\cdot diam(\mathcal{D}))^p\cdot exp{\{ -\frac{p}{4}({12log(2n+2)})+ p\cdot log(\sqrt{2n}) \  \}} +4\cdot12^2p^2l^4_{\mathcal{D}}(\frac{log(2n+2)}{n})^2 \nonumber \\
	\leq&(2)^{4-2.5p}(l_{\mathcal{D}})^4(2\cdot diam(\mathcal{D}))^p\cdot \frac{1}{(2n+2)^{2p}} +4\cdot12^2p^2l^4_{\mathcal{D}}(\frac{log(2n+2)}{n})^2\nonumber\\
	 \leq& C_{\mathcal{D}}^{\prime\prime}\big( \frac{log (2n+2)}{n}\big)^2 \leq C_{\mathcal{D}}^{\prime\prime}\big( \frac{log (4n)}{n}\big)^2
	\end{align}
	where $C_{\mathcal{D}}^{\prime\prime} \triangleq (2)^{4-2.5p}(l_{\mathcal{D}})^4(2\cdot diam(\mathcal{D}))^p+4\cdot12^2p^2l^4_{\mathcal{D}}(\frac{log(2n+2)}{n})^2$. The second inequality follows from that $(x-1)^2 \geq \frac{x^2}{4}$ when $x \geq 2$. The third inequality follows from $-3p\cdot log(2n+2)+\frac{p}{2}\cdot log(2n) \leq -2p\cdot log(2n+2)$. The last two inequalities follows from $p \geq 1 $ and $n \geq 1$.
	
	Finally, to prove \eqref{t4} and  \eqref{z4}, we notice that $g_w(x)$ and its second order derivatives are also $l_{\mathcal{D}}$-Lipschitz continous for all $w \in \Omega_w$ according to Assumption \ref{lp} and Definition \ref{ld}. Thus \eqref{28} becomes:
	\begin{align*}&\sup\limits_{\substack{x,\tilde{x} \in \mathcal{D}, \\ 1 \leq k \leq d}} \{\lvert {\partial [g_{w}]_k}(x)-{\partial [g_{w}]_k}(\tilde{x})\rvert \} \leq l_{\mathcal{D}}\cdot \sqrt{p}\|{x-\tilde{x}}\|_{\infty}\nonumber\\
	&\sup\limits_{\substack{x,\tilde{x} \in \mathcal{D}, \\ 1 \leq i,j \leq p \\ 1 \leq k \leq d}} \{\lvert \frac{\partial [g_{w}]_k}{\partial[x]_i\partial[x]_j}(x)-\frac{\partial [g_{w}]_k}{\partial[x]_i\partial[x]_j}(\tilde{x})\rvert \} \leq l_{\mathcal{D}}\cdot \sqrt{p}\|{x-\tilde{x}}\|_{\infty} \nonumber\\,
	\end{align*}
	for equation \eqref{t4} and \eqref{z4}, respectively. The rest of the proof follows similarly.
\end{proof}
Now we proceed with the main theorem of this section.
\begin{Theorem}\label{th3}
There exist a constant $C_{\mathcal{D}}< \infty$ such that for any $v_1 \in \Omega_v$ and $x,\tilde{x} \in \mathcal{D}$, the $W(x,v_1)$ and $W(\tilde{x},v_1)$ from the variance reduced unbiased gradient  $W=W(x,v_1)-W(\tilde{x},v_1)+g(\tilde{x})$ in Algorithm 3 satisfies:
\begin{equation}
\mathbb{E}\|W(x,v_1)-W(\tilde{x},v_1)\|_2^2 \leq C_{\mathcal{D}}\|x-\tilde{x}\|_2^2
\end{equation}
\end{Theorem}
\begin{proof}
Fix $v_1\in \Omega_v$ and $x,\tilde{x} \in \mathcal{D}$, we have 
\begin{align*}
W(x,v_1)-W(\tilde{x},v_1)=\frac{1}{\tilde{p}_N}\bigg(Y_1(x)-Y_1(\tilde{x})-0.5\cdot\Big(Y_2(x)-Y_2(\tilde{x})+Y_3(x)-Y_3(\tilde{x})\Big)\bigg)+Y_4(x)-Y_4(\tilde{x})
\end{align*}
and it follows that 
\begin{align}\label{lpbound}
&\mathbb{E}\|W(x,v_1)-W(\tilde{x},v_1)\|_2^2 \nonumber\\
=&\sum_{n=0}^{\infty} \mathbb{E}[\|W(x,v_1)-W(\tilde{x},v_1)\|_2^2|N=n]\cdot \tilde{p}_n  \nonumber \\
=&\sum_{n=0}^{\infty}\sum_{i=1}^p \mathbb{E}[([W(x,v_1)]_i-[W(\tilde{x},v_1)]_i)^2|N=n]\cdot \tilde{p}_n   \nonumber \\
\leq &\sum_{i = 1}^p2\mathbb{E}([Y_4(x)]_i-[Y_4(\tilde{x})]_i)^2+\sum_{n=0}^{\infty} \sum_{i = 1}^p \frac{2}{\tilde{p}_n}\mathbb{E}[\bigg([Y_1(x)]_i-[Y_1(\tilde{x})]_i-0.5\cdot\Big([Y_2(x)]_i-[Y_2(\tilde{x})]_i+[Y_3(x)]_i-[Y_3(\tilde{x})]_i\Big)\bigg)^2|N=n] \nonumber\\
=&\sum_{i = 1}^p2\mathbb{E}(\nabla [Y_4(\rho_i)]_i^\intercal (x-\tilde{x}))^2+\sum_{n=0}^{\infty} \sum_{i = 1}^p \frac{2}{\tilde{p}_n} \cdot\mathbb{E} [\bigg(\nabla \Big([Y_1(\xi_i)]_i-0.5\cdot([Y_2(\xi_i)]_i+[Y_3(\xi_i)]_i)\Big)^\intercal(x-\tilde{x})\bigg)^2|N=n]\nonumber\\
=& \sum_{i = 1}^p2\|x-\tilde{x}\|_2^2\cdot\mathbb{E}\|\nabla [Y_4(\rho_i)]_i \|_2^2+ \sum_{n=0}^{\infty} \sum_{i = 1}^p \frac{2\|x-\tilde{x}\|_2^2}{\tilde{p}_n} \cdot \mathbb{E}[\|\nabla \Big([Y_1(\xi_i)]_i-0.5\cdot([Y_2(\xi_i)]_i+[Y_3(\xi_i)]_i)\Big)\|_2^2\quad|N=n]
\end{align}
The last two lines follows from Mean Value Theorem where $\xi_i,\rho_i, 1 \leq i \leq p$ lie somewhere between $x$ and $\tilde{x}$. To proceed with equation \eqref{lpbound}, fix $N=n$ where $n \geq 0$, notice we can write
 $Y_1(\xi_i)=[\bar{S}_{2^{n+n_0+1}}(\xi_i)]^\intercal \cdot \nabla f_{v_1}(\bar{T}_{2^{n+n_0+1}}(\xi_i))$. Thus we have
\begin{align*} [Y_1(\xi_i)]_i=&\sum\limits_{k=1}^d[\bar{S}_{2^{n+n_0+1}}(\xi_i)]_{ki}\cdot[\nabla f_{v_1}(\bar{T}_{2^{n+n_0+1}}(\xi_i))]_k\nonumber\\
=&\sum\limits_{k=1}^d(\frac{1}{2^{n+n_0+1}}\cdot\sum_{t=1}^{2^{n+n_0+1}} \frac{\partial[g_{w_t}]_k}{\partial [x]_i}(\xi_i)) \cdot \frac{\partial f_{v_1}}{\partial [y]_k} (\frac{1}{2^{n+n_0+1}}\cdot \sum_{t=1}^{2^{n+n_0+1}}g_{w_t}(\xi_i)),
\end{align*}
and it follows from the chain rule that, for $ 1 \leq j \leq p$,
\begin{align}\label{lpgradient}
&[\nabla [Y_1(\xi_i)]_i]_j \nonumber\\
=&\sum_{k=1}^d\bigg((\frac{1}{2^{n+n_0+1}}\cdot\sum_{t=1}^{2^{n+n_0+1}} \frac{\partial[g_{w_t}]_k}{\partial [x]_i\cdot \partial[x]_j}(\xi_i)) \cdot \frac{\partial f_{v_1}}{\partial [y]_k} (\frac{1}{2^{n+n_0+1}}\cdot \sum_{t=1}^{2^{n+n_0+1}}g_{w_t}(\xi_i)) \nonumber\\
&+(\frac{1}{2^{n+n_0+1}}\cdot\sum_{t=1}^{2^{n+n_0+1}} \frac{\partial[g_{w_t}]_k}{\partial [x]_i}(\xi_i)) \cdot \Big(\sum_{h=1}^d \frac{\partial f_{v_1}}{\partial [y]_k\cdot \partial[y]_h} (\frac{1}{2^{n+n_0+1}}\cdot \sum_{t=1}^{2^{n+n_0+1}}g_{w_t}(\xi_i))\cdot(\frac{1}{2^{n+n_0+1}}\cdot\sum_{t=1}^{2^{n+n_0+1}} \frac{\partial[g_{w_t}]_h}{\partial [x]_j}(\xi_i)) \Big)\bigg)\nonumber\\
=&\sum_{k=1}^d\bigg([\bar{Z}_{2^{n+n_0+1}}(\xi_i)]_{kij}\cdot\frac{\partial f_{v_1}}{\partial [y]_k}(\bar{T}_{2^{n+n_0+1}}(\xi_i))+[\bar{S}_{2^{n+n_0+1}}(\xi_i)]_{ki}\cdot\Big(\sum_{h=1}^d\frac{\partial f_{v_1}}{\partial [y]_k\cdot \partial[y]_h}(\bar{T}_{2^{n+n_0+1}}(\xi_i))\cdot [\bar{S}_{2^{n+n_0+1}}(\xi_i)]_{hj}  \Big)\bigg),
\end{align}
where the last line follows from  Definition \ref{sandt}.It follows from the definition of $l_{\mathcal{D}}$ that for any $\xi_i \in \mathcal{D}$ and $N=n$, $|[\nabla[Y_1(\xi_i)]_i]_j|\leq dl^2_{\mathcal{D}}(1+dl_{\mathcal{D}})$ and $\|\nabla[Y_1(\xi_i)]_i\|^2_2 \leq pd^2l^4_{\mathcal{D}}(1+dl_{\mathcal{D}})^2$. Following a similar analysis, we also have $\|\nabla[Y_4(\rho_i)]_i\|^2_2 \leq pd^2l^4_{\mathcal{D}}(1+dl_{\mathcal{D}})^2$, so the first term of \eqref{lpbound} satisfies:
\begin{equation}\label{firstterm} 2\|x-\tilde{x}\|_2^2\cdot\mathbb{E}\|\nabla [Y_4(\rho_i)]_i \|_2^2 \leq 2pd^2l^4_{\mathcal{D}}(1+dl_{\mathcal{D}})^2\cdot\|x-\tilde{x}\|_2^2.
\end{equation}

To bound the second term in \eqref{lpbound}, we define the following function:
for $x \in \mathcal{H}\subseteq \mathbb{R}^{d \times p },y \in \mathcal{G}\subseteq \mathbb{R}^{d }, z \in \mathcal{J}\subseteq \mathbb{R}^{d \times p \times p}$ and each $1 \leq i,j 
\leq p$, define $G:\mathcal{H}\times\mathcal{G}\times\mathcal{J} \rightarrow \mathbb{R}$ by:
\begin{equation}
G_{ij}(x,y,z)\triangleq\sum_{k=1}^d\bigg([z]_{kij}\cdot\frac{\partial f_{v_1}}{\partial [y]_k}(y)+[x]_{ki}\cdot\Big(\sum_{h=1}^d\frac{\partial f_{v_1}}{\partial [y]_k\cdot \partial[y]_h}(y)\cdot [x]_{hj}  \Big)\bigg),
\end{equation}
It is straightforward to see that for any realization of $N$, $[\nabla[Y_1(x)]_i]_j=G_{ij}(\bar{S}_{2^{N+n_0+1}}(x),\bar{T}_{2^{N+n_0+1}}(x),\bar{Z}_{2^{N+n_0+1}}(x))$. We can compute each component of the gradient $\nabla G(x,y,z) \in \mathbb{R}^{(d\times p)\times d \times (d\times p \times p)}$ as
\begin{align}
\frac{\partial G_{ij}}{\partial[x]_{k^{\prime}j^{\prime}}}=& \delta_{ij^{\prime}} \cdot \sum_{h=1}^d\frac{\partial f_{v_1}}{\partial [y]_{k^{\prime}}\cdot \partial[y]_h}(y)\cdot [x]_{hj} +\delta_{jj^{\prime}} \cdot \sum_{k=1}^d  \frac{\partial f_{v_1}}{\partial [y]_k\cdot \partial[y]_{k^{\prime}}}(y) \cdot [x]_{ki} \nonumber\\
\frac{\partial G_{ij}}{\partial[y]_{h^{\prime}}}=& \sum_{k=1}^d\bigg([z]_{kij}\cdot\frac{\partial f_{v_1}}{\partial [y]_k\partial[y]_{h^{\prime}}}(y)+[x]_{ki}\cdot\Big(\sum_{h=1}^d\frac{\partial f_{v_1}}{\partial [y]_k \partial[y]_h \partial[y]_{h^{\prime}}}(y)\cdot [x]_{hj}  \Big)\bigg) \nonumber\\
\frac{\partial G_{ij}}{\partial[z]_{k^{\prime}i^{\prime}j^{\prime}}}=& \delta_{ii^{\prime}}\cdot \delta_{jj^{\prime}}\cdot \frac{\partial f_{v_1}}{\partial[y]_{k^{\prime}}}(y)
\end{align}
where $1 \leq i^{\prime}, j^{\prime} \leq p$,$1 \leq k^{\prime}, h^{\prime} \leq d$ and $\delta_{ij}$ is the Kronecker delta. It follows from Assumptions 1, Assumption 3 and Lemma \eqref{lipschitz} that for any $1 \leq i,j \leq p$, each component of the $\nabla G_{ij}(x,y,z)$ is Lipschitz continous with constant $2dl^2_{\mathcal{D}}(1+dl_{\mathcal{D}})$ and $\nabla G_{ij}(x,y,z)$ is Lipschitz continous with constant $2\sqrt{dp^2+dp+d}\cdot dl^2_{\mathcal{D}}(1+dl_{\mathcal{D}})$. Thus if we define  $R_{ij}\begin{pmatrix*} x,x_0\\y,y_0\\z,z_0
\end{pmatrix*}\triangleq G_{ij}(x,y,z)-G_{ij}(x_0,y_0,z_0)-\nabla G_{ij}(x_0,y_0,z_0)\cdot\begin{pmatrix*} x-x_0\\y-y_0\\z-z_0
\end{pmatrix*}$ for $1 \leq i,j\leq p$, using Lemma \ref{secondorder},
\begin{align}\label{reminder2}
&|R_{ij}\begin{pmatrix*} x,x_0\\y,y_0\\z,z_0
\end{pmatrix*}| 
=|G(x,y,z)-G(x_0,y_0,z_0)-\nabla G(x_0,y_0,z_0) \cdot  \begin{pmatrix*} x-x_0\\y-y_0\\z-z_0
\end{pmatrix*}|^2 \nonumber\\
\leq & \sqrt{dp^2+dp+d}\cdot dl^2_{\mathcal{D}}(1+dl_{\mathcal{D}}) (\sum_{k^{\prime}=1}^d\sum_{j^{\prime}=1}^p ([x]_{k^{\prime}j^{\prime}}-[x_0]_{k^{\prime}j^{\prime}})^2+\sum_{h^{\prime}=1}^d ([y]_{h^{\prime}}-[y_0]_{h^{\prime}})^2+\sum_{k^{\prime}=1}^d\sum_{i^{\prime},j^{\prime}=1}^p ([z]_{k^{\prime}i^{\prime}j^{\prime}}-[z_0]_{k^{\prime}i^{\prime}j^{\prime}})^2) \nonumber\\
\leq & 4p d^3l^3_{\mathcal{D}} (\sum_{k^{\prime}=1}^d\sum_{j^{\prime}=1}^p ([x]_{k^{\prime}j^{\prime}}-[x_0]_{k^{\prime}j^{\prime}})^2+\sum_{h^{\prime}=1}^d ([y]_{h^{\prime}}-[y_0]_{h^{\prime}})^2+\sum_{k^{\prime}=1}^d\sum_{i^{\prime},j^{\prime}=1}^p ([z]_{k^{\prime}i^{\prime}j^{\prime}}-[z_0]_{k^{\prime}i^{\prime}j^{\prime}})^2)
\end{align}
for any $x,x_0 \in \mathcal{H}$,$y,y_0 \in \mathcal{G}$ and $z,z_0 \in \mathcal{J}$. To bound the second term in \eqref{lpbound}, for any $n \geq 0$ and $1 \leq i\leq p$,
\begin{align}\label{almostdone}
& \mathbb{E}[\|\nabla \Big([Y_1(\xi_i)]_i-0.5\cdot([Y_2(\xi_i)]_i+[Y_3(\xi_i)]_i)\Big)\|_2^2|N=n] \nonumber \\
=&\sum_{j=1}^p\mathbb{E}[\big([\nabla[Y_1(\xi_i)]_i]_j-0.5\cdot ([\nabla[Y_2(\xi_i)]_i]_j+[\nabla[Y_3(\xi_i)]_i]_j)\big)^2|N=n] \nonumber\\
=&\sum_{j=1}^p\mathbb{E} \Big[\Big(G_{ij}(\bar{S}_{2^{N+n_0+1}}(\xi_i),\bar{T}_{2^{N+n_0+1}}(\xi_i),\bar{Z}_{2^{N+n_0+1}}(\xi_i))\nonumber\\
&-0.5\cdot G_{ij}(\bar{S}_{2^{N+n_0}}(\xi_i),\bar{T}_{2^{N+n_0}}(\xi_i),\bar{Z}_{2^{N+n_0}}(\xi_i))-0.5\cdot G_{ij}(\tilde{S}_{2^{N+n_0}}(\xi_i),\tilde{T}_{2^{N+n_0}}(\xi_i),\tilde{Z}_{2^{N+n_0}}(\xi_i))\Big)^2|N=n\Big] \nonumber\\ 
=& \sum_{j=1}^p\mathbb{E}\Big[\Big(G_{ij}(\mathbb{E}S(\xi_i),\mathbb{E}T(\xi_i),\mathbb{E}Z(\xi_i))+\nabla G_{ij}(\mathbb{E}S(\xi_i),\mathbb{E}T(\xi_i),\mathbb{E}Z(\xi_i))\cdot\begin{pmatrix*}
\bar{S}_{2^{N+n_0+1}}(\xi_i)-\mathbb{E}S(\xi_i)\\\bar{T}_{2^{N+n_0+1}}(\xi_i)-\mathbb{E}T(\xi_i)\\\bar{Z}_{2^{N+n_0+1}}(\xi_i)-\mathbb{E}Z(\xi_i)
\end{pmatrix*}\nonumber\\
&-G_{ij}(\mathbb{E}S(\xi_i),\mathbb{E}T(\xi_i),\mathbb{E}Z(\xi_i))-\nabla G_{ij}(\mathbb{E}S(\xi_i),\mathbb{E}T(\xi_i),\mathbb{E}Z(\xi_i))\cdot\begin{pmatrix*}
\frac{\bar{S}_{2^{N+n_0}}(\xi_i)+\tilde{S}_{2^{N+n_0}}(\xi_i)}{2}-\mathbb{E}S(\xi_i)\\\frac{\bar{T}_{2^{N+n_0}}(\xi_i)+\tilde{T}_{2^{N+n_0}}(\xi_i)}{2}-\mathbb{E}T(\xi_i)\\\frac{\bar{Z}_{2^{N+n_0}}(\xi_i)+\tilde{Z}_{2^{N+n_0}}(\xi_i)}{2}-\mathbb{E}Z(\xi_i)
\end{pmatrix*}\nonumber\\
&+R_{ij}\begin{pmatrix*}
\bar{S}_{2^{N+n_0+1}}(\xi_i),\mathbb{E}S(\xi_i)\\\bar{T}_{2^{N+n_0+1}}(\xi_i),\mathbb{E}T(\xi_i)\\\bar{Z}_{2^{N+n_0+1}}(\xi_i),\mathbb{E}Z(\xi_i)
\end{pmatrix*} -\frac{1}{2}R_{ij}\begin{pmatrix*}
\bar{S}_{2^{N+n_0}}(\xi_i),\mathbb{E}S(\xi_i)\\\bar{T}_{2^{N+n_0}}(\xi_i),\mathbb{E}T(\xi_i)\\\bar{Z}_{2^{N+n_0}}(\xi_i),\mathbb{E}Z(\xi_i)
\end{pmatrix*}-\frac{1}{2}R_{ij}\begin{pmatrix*}
\tilde{S}_{2^{N+n_0}}(\xi_i),\mathbb{E}S(\xi_i)\\\tilde{T}_{2^{N+n_0}}(\xi_i),\mathbb{E}T(\xi_i)\\\tilde{Z}_{2^{N+n_0}}(\xi_i),\mathbb{E}Z(\xi_i)
\end{pmatrix*}\Big)^2|N=n\Big]\nonumber\\
=& \sum_{j=1}^p \mathbb{E} \Big[(R_{ij}\begin{pmatrix*}
\bar{S}_{2^{N+n_0+1}}(\xi_i),\mathbb{E}S(\xi_i)\\\bar{T}_{2^{N+n_0+1}}(\xi_i),\mathbb{E}T(\xi_i)\\\bar{Z}_{2^{N+n_0+1}}(\xi_i),\mathbb{E}Z(\xi_i)
\end{pmatrix*} -\frac{1}{2}R_{ij}\begin{pmatrix*}
\bar{S}_{2^{N+n_0}}(\xi_i),\mathbb{E}S(\xi_i)\\\bar{T}_{2^{N+n_0}}(\xi_i),\mathbb{E}T(\xi_i)\\\bar{Z}_{2^{N+n_0}}(\xi_i),\mathbb{E}Z(\xi_i)
\end{pmatrix*}-\frac{1}{2}R_{ij}\begin{pmatrix*}
\tilde{S}_{2^{N+n_0}}(\xi_i),\mathbb{E}S(\xi_i)\\\tilde{T}_{2^{N+n_0}}(\xi_i),\mathbb{E}T(\xi_i)\\\tilde{Z}_{2^{N+n_0}}(\xi_i),\mathbb{E}Z(\xi_i)
\end{pmatrix*})^2|N=n\Big] \nonumber\\
\leq&  \frac{3}{4}\sum_{j=1}^p \mathbb{E}\Big[4R^2_{ij}\begin{pmatrix*}
\bar{S}_{2^{N+n_0+1}}(\xi_i),\mathbb{E}S(\xi_i)\\\bar{T}_{2^{N+n_0+1}}(\xi_i),\mathbb{E}T(\xi_i)\\\bar{Z}_{2^{N+n_0+1}}(\xi_i),\mathbb{E}Z(\xi_i)
\end{pmatrix*} +R^2_{ij}\begin{pmatrix*}
\bar{S}_{2^{N+n_0}}(\xi_i),\mathbb{E}S(\xi_i)\\\bar{T}_{2^{N+n_0}}(\xi_i),\mathbb{E}T(\xi_i)\\\bar{Z}_{2^{N+n_0}}(\xi_i),\mathbb{E}Z(\xi_i)
\end{pmatrix*}+R^2_{ij}\begin{pmatrix*}
\tilde{S}_{2^{N+n_0}}(\xi_i),\mathbb{E}S(\xi_i)\\\tilde{T}_{2^{N+n_0}}(\xi_i),\mathbb{E}T(\xi_i)\\\tilde{Z}_{2^{N+n_0}}(\xi_i),\mathbb{E}Z(\xi_i)
\end{pmatrix*}|N=n \Big] \nonumber\\
\leq&\sum_{j=1}^p 12p^2 d^6l^6_{\mathcal{D}}(d+dp+dp^2)\|x-\tilde{x}\|_2^2 \cdot \nonumber\\
& \mathbb{E}\Big[\sum_{h^{\prime}=1}^d 4([\bar{T}_{2^{N+n_0+1}}(\xi_i)]_{h^{\prime}}-[\mathbb{E}T(\xi_i)]_{h^{\prime}})^4+([\bar{T}_{2^{N+n_0}}(\xi_i)]_{h^{\prime}}-[\mathbb{E}T(\xi_i)]_{h^{\prime}})^4+([\tilde{T}_{2^{N+n_0}}(\xi_i)]_{h^{\prime}}-[\mathbb{E}T(\xi_i)]_{h^{\prime}})^4 \nonumber\\
&+\sum_{\substack{1 \leq k^{\prime} \leq d \\1 \leq i^{\prime} \leq p\\1 \leq j^{\prime} \leq p }} 4([\bar{Z}_{2^{N+n_0+1}}(\xi_i)]_{k^{\prime}i^{\prime}j^{\prime}}-[\mathbb{E}Z(\xi_i)]_{k^{\prime}i^{\prime}j^{\prime}})^4+([\bar{Z}_{2^{N+n_0}}(\xi_i)]_{k^{\prime}i^{\prime}j^{\prime}}-[\mathbb{E}Z(\xi_i)]_{k^{\prime}i^{\prime}i^{\prime}j^{\prime}})^4+([\tilde{Z}_{2^{N+n_0}}(\xi_i)]_{k^{\prime}i^{\prime}j^{\prime}}-[\mathbb{E}Z(\xi_i)]_{k^{\prime}i^{\prime}j^{\prime}})^4\nonumber\\
&+\sum_{\substack{1 \leq k^{\prime} \leq d \\1 \leq j^{\prime} \leq p }} 4([\bar{S}_{2^{N+n_0+1}}(\xi_i)]_{k^{\prime}j^{\prime}}-[\mathbb{E}S(\xi_i)]_{k^{\prime}j^{\prime}})^4+([\bar{S}_{2^{N+n_0}}(\xi_i)]_{k^{\prime}j^{\prime}}-[\mathbb{E}S(\xi_i)]_{k^{\prime}j^{\prime}})^4+([\tilde{S}_{2^{N+n_0}}(\xi_i)]_{k^{\prime}j^{\prime}}-[\mathbb{E}S(\xi_i)]_{k^{\prime}j^{\prime}})^4 |N=n\Big],
\end{align}
where the last two inequality follows from equation \eqref{elementary} and  \eqref{reminder2}. The equality above it  follows from equation \eqref{SandT}. Continuing on \eqref{almostdone}, it follows from Lemma \ref{finiteclt} that 
\begin{align}\label{blah2}
\mathbb{E}[\|\nabla \Big([Y_1(\xi_i)]_i-0.5\cdot([Y_2(\xi_i)]_i+[Y_3(\xi_i)]_i)\Big)\|_2^2|N=n] \leq&  36p^3 d^6l^6_{\mathcal{D}}(d+dp+dp^2) C_{\mathcal{D}}^{\prime\prime}\cdot(log2)^2\cdot\frac{(n+n_0+3)^2}{2^{2(n+n_0)}}\cdot \|x-\tilde{x}\|_2^2 \nonumber\\
\leq&108p^5 d^7l^6_{\mathcal{D}} C_{\mathcal{D}}^{\prime\prime}\cdot(log2)^2\cdot\frac{(n+n_0+3)^2}{2^{2(n+n_0)}}\cdot \|x-\tilde{x}\|_2^2
\end{align}
Combine \eqref{firstterm} and \eqref{blah2}. Let $C_{\mathcal{D}} \triangleq 2p^2d^2l^4_{\mathcal{D}}(1+dl_{\mathcal{D}})^2+\frac{216p^6 d^7l^6_{\mathcal{D}} C_{\mathcal{D}}^{\prime\prime}\cdot(log2)^2}{(1-0.5^{\gamma})2^{2n_0}}\sum\limits_{n=0}^{\infty}{(n+n_0+3)^2}\cdot{2^{(\gamma-2)n}} $. Notice $C_{\mathcal{D}} < \infty$ for any $n_0 \geq 0$ because $\gamma<2$.  Now \eqref{lpbound} becomes:
\begin{align*}
\mathbb{E}\|W(x,v_1)-W(\tilde{x},v_1)\|_2^2 \leq& \Big(\sum_{i=1}^p  2pd^2l^4_{\mathcal{D}}(1+dl_{\mathcal{D}})^2+\sum_{i=1}^p\sum_{n=0}^{\infty}216p^5 d^7l^6_{\mathcal{D}} C_{\mathcal{D}}^{\prime\prime}\cdot(log2)^2\cdot\frac{(n+n_0+3)^2}{2^{2(n+n_0)}\tilde{p}_n}\Big)\cdot\|x-\tilde{x}\|_2^2 \nonumber\\
=& C_{\mathcal{D}}\cdot\|x-\tilde{x}\|_2^2
\end{align*}
\end{proof}
\subsection{Properties of the Simulated Variance Reduced Gradient Algorithm}
In this section we prove the convergence property of Algorithm 4. Notice the constant $C_{\mathcal{D}}$ is defined in Theorem \ref{th3} and $\mu$ is the strong convexity coefficient.
\begin{Lemma}\label{coco}
Let $F:\mathbb{R}^p \rightarrow \mathbb{R}$ be a convex function with $L$-Lipschitz gradient and denote $x_{\star}=\arg\min\limits_{x \in \mathbb{R}^p} F(x)$ to be the global minimizer of $F(\cdot)$, then for any $x \in \mathbb{R}^p$,
\begin{equation*}
\frac{1}{2L}\|\nabla F(x)\|_2^2 \leq F(x)-F(x_{\star}).
\end{equation*}
\end{Lemma}
\begin{proof}
	Let $F_x(y)=F(x)+\nabla F(x)(y-x)+\frac{L}{2}\|y-x\|_2^2$, since $F(\cdot)$ has $L$-Lipschitz gradient, we have $F(y) \leq F_x(y)$ for all $x \in \mathbb{R}^p$. It then follows $F(x_{\star}) \leq \min\limits_{y \in \mathbb{R}^p} F_x(y)$. It is straightforward to compute the global minimizer $y_{\star}$ of the quadratic function $F_x(y)$ to be $y_{\star}=x-\frac{1}{L}\nabla F(x)$.so we have:
	\begin{equation*}
	F(x_{\star}) \leq  F_x(y_{\star})= F(x)-\frac{1}{2L}\|F(x)\|_2^2
	\end{equation*}
\end{proof}
\begin{Theorem}\label{Big1}
	Consider the Simulated SVRG Algorithm 4 with options II. Let $\lambda$ is small and $M$ is sufficiently large so that 
	\begin{equation}
	\alpha = \frac{1}{\mu(1-\frac{4}{\mu}C_{\mathcal{D}}\lambda)\lambda M}+\frac{(\frac{4}{\mu}C_{\mathcal{D}}+2L)\lambda}{1-\frac{4}{\mu}C_{\mathcal{D}}\lambda} < 1.
	\end{equation} Then under  Assumptions 1-3, we have geometric convergence in expectation for the Simulated SVRG :
	\begin{equation*}
	\mathbb{E}[F(\tilde{x}_s)] \leq F(\tilde{x}_{\star})+\alpha^s[F(\tilde{x}_0)-F(\tilde{x}_{\star})] 
	\end{equation*}
\end{Theorem}
\begin{proof}
	It follows from Lemma \ref{coco} that
	\begin{equation}\label{svrg8}
	\|\nabla F(x)-\nabla F(x_{\star})\|^2_2=\|\nabla F(x)\|^2_2 \leq 2L [F(x)-F(x_{\star})]
	\end{equation}
	Now conditioning on $x_{t-1}$, we can take expectation with respect to $v_t \in \Omega_v$ to obtain 
	\begin{align}
	\mathbb{E}[\|\nu_t\|_2^2 \ | \  x_{t-1}] \leq & 2\mathbb{E}[\|W(x_{t-1},v_t)-W(\tilde{x},v_t)\|_2^2\ | \  x_{t-1}] + 2 \nabla \|F(\tilde{x}) \|_2^2  \nonumber \\
	\leq & 2C_{\mathcal{D}}\|x_{t-1}-\tilde{x}\|_2^2+ 4L[F(\tilde{x})-F(x_{\star})] \nonumber \\
	\leq & 4C_{\mathcal{D}}(\|x_{t-1}-{x_{\star}}\|_2^2 +\|\tilde{x}-x_{\star}\|_2^2) + 4L[F(\tilde{x})-F(x_{\star})] \nonumber \\
	\leq & \frac{8}{\mu}C_{\mathcal{D}} \cdot [F(x_{t-1})-F(x_{\star})] + (\frac{8}{\mu}C_{\mathcal{D}} + 4L)\cdot[F(\tilde{x})-F(x_{\star})].
	\end{align}
	where the second inequality follows from Theorem \ref{th3} and equation \eqref{svrg8}. The last inequality follows from the strong convexity of $F(\cdot)$. Thus,
	\begin{align}
	&\mathbb{E}[\|x_t-x_{\star}\|_2^2 \ | \ x_{t-1}] \nonumber \\
	=&\|x_{t-1}-x_{\star}\|_2^2-2\lambda(x_{t-1}-x_{\star})^{\intercal}\mathbb{E}[\nu_t \ | x_{t-1}]+\lambda^2\mathbb{E}[\|\nu_t\|_2^2 \ | x_{t-1}] \nonumber \\
	\leq & \|x_{t-1}-x_{\star}\|_2^2-2\lambda(x_{t-1}-x_{\star})^{\intercal}\nabla F(x_{t-1})+ \frac{8}{\mu}C_{\mathcal{D}} \lambda^2\cdot [F(x_{t-1})-F(x_{\star})] + (\frac{8}{\mu}C_{\mathcal{D}} + 4L)\lambda^2\cdot[F(\tilde{x})-F(x_{\star})]\nonumber \\
	\leq & \|x_{t-1}-x_{\star}\|_2^2-2\lambda [F(x_{t-1})-F(x_{\star})]+\frac{8}{\mu}C_{\mathcal{D}} \lambda^2\cdot [F(x_{t-1})-F(x_{\star})] + (\frac{8}{\mu}C_{\mathcal{D}} + 4L)\lambda^2\cdot[F(\tilde{x})-F(x_{\star})] \nonumber \\
	=& \|x_{t-1}-x_{\star}\|_2^2-2\lambda(1-\frac{4}{\mu}C_{\mathcal{D}}\lambda) [F(x_{t-1})-F(x_{\star})]+(\frac{8}{\mu}C_{\mathcal{D}} + 4L)\lambda^2[F(\tilde{x})-F(x_{\star})].
	\end{align}
	where the third line follows from the unbiasedness of the simulated gradient and the fourth line follows from the convexity of $F(\cdot)$. Now we consider a fixed stage $s$, so that $\tilde{x}=\tilde{x}_{s-1}$ and $\tilde{x}_s$ is selected uniformly after all $M$ updates are completed. Summing over the previous inequality over $t=1,...,M$, taking expectation and use options II at stage $s$, we obtain
	\begin{align}
	&\mathbb{E}[\|x_{M}-x_{\star}\|_2^2]+2\lambda(1-\frac{4}{\mu}C_{\mathcal{D}}\lambda)M\mathbb{E}[F(\tilde{x}_s)-F(x_{\star})] \nonumber \\
	\leq & \mathbb{E}[\|x_{0}-x_{\star}\|_2^2] +(\frac{8}{\mu}C_{\mathcal{D}} + 4L)\lambda^2M\mathbb{E}[F(\tilde{x})-F(x_{\star})] \nonumber \\
	= & \mathbb{E}[\|\tilde{x}-x_{\star}\|_2^2] + (\frac{8}{\mu}C_{\mathcal{D}} + 4L)\lambda^2M\mathbb{E}[F(\tilde{x})-F(x_{\star})] \nonumber \\
	\leq & \frac{2}{\mu} \mathbb{E}[F(\tilde{x})-F(x_{\star})]+(\frac{8}{\mu}C_{\mathcal{D}} + 4L)\lambda^2M\mathbb{E}[F(\tilde{x})-F(x_{\star})] \nonumber \\
	= & (\frac{2}{\mu}+(\frac{8}{\mu}C_{\mathcal{D}} + 4L)\lambda^2 M)\mathbb{E}[F(\tilde{x})-F(x_{\star})]
	\end{align}
	Thus we obtain
	\begin{equation}
	\mathbb{E}[F(\tilde{x}_s)-F(x_{\star})] \leq \bigg[\frac{1}{\mu(1-\frac{4}{\mu}C_{\mathcal{D}}\lambda)\lambda M}+\frac{(\frac{4}{\mu}C_{\mathcal{D}}+2L)\lambda}{1-\frac{4}{\mu}C_{\mathcal{D}}\lambda}\bigg]\mathbb{E}[F(\tilde{x}_{s-1})-F(x_{\star})]
	\end{equation}
	This implies that $\mathbb{E}[F(\tilde{x}_{s})-F(x_{\star})] \leq \alpha^s \cdot \mathbb{E}[F(\tilde{x}_{0})-F(x_{\star})]$. The conclusion follows.
\end{proof}
\begin{Corollary}
Let $\{\tilde{x}_s\}_{s\geq 0}$ be the sequence of output from each epoch of the Simulated SVRG algorithm. Then, with probability 1, $\tilde{x}_s$ converge exponentially fast to $x_{\star}$.
\end{Corollary}
\begin{proof}
It follows from Theorem \ref{Big1} that we can find $0<\alpha<1$ such that $
\mathbb{E}[F(\tilde{x}_s)] \leq F(\tilde{x}_{\star})+\alpha^s[F(\tilde{x}_0)-F(\tilde{x}_{\star})] 
$.
Pick any $\alpha<\rho<1$. Define the set  $\mathcal{A}_{s}=\{F(\tilde{x}_s)-F(x_{\star})>\rho^s\}$ in probability space, we have $\mathbb{P}(\mathcal{A}_s) \leq (\frac{\alpha}{\rho})^s\cdot\mathbb{E}[F(\tilde{x}_{0})-F(x_{\star})]$ which implies that $\sum_{s\geq 0}\mathbb{P}(\mathcal{A}_s) < \infty$. It then follows from Borel-Cantelli lemma that 
\begin{align}\label{a.s.}
\mathbb{P}(\mathcal{A}_s \text{ occurs infinitely often} )=\mathbb{P}\Big(\limsup_{s\rightarrow \infty}\mathcal{A}_s\Big)=\mathbb{P}(\bigcap_{t=0}^{\infty}\bigcup_{s=t}^{\infty}\mathcal{A}_s) =\inf_{t\geq 0}\mathbb{P}(\bigcup_{s=t}^{\infty}\mathcal{A}_s) \leq \inf_{t\geq 0}\sum_{s \geq t}\mathbb{P}(\mathcal{A}_s) =0.
\end{align}
Thus with probability 1, $F(\tilde{x}_s)-F(x_{\star})< \rho^s$ for $s$ large enough which implies $\|\tilde{x}_s-x_{\star}\|_2^2 \leq \frac{2}{\mu}\rho^s$.
\end{proof}
\subsection{Properties of the Stochastically Controlled Simulated Gradient Algorithm}
In this section we prove the convergence property of Algorithm 5.
\begin{Lemma}
	Fix $x \in \mathcal{D}$ and $K,B \geq 1$, we sample a batch $\mathcal{I} \subset \Omega_v$ with $|{\mathcal{I}}|=B$ following the distribution of $v$ and independently generate $h_k(x)=\frac{1}{B}\sum_{v_i\in \mathcal{I}} \text{UnibasedGradient}(x,v_i)$ for $1\leq k \leq K$. Let $C^{\prime}_{\mathcal{D}}$ be the constant in the proof of Proposition \ref{finitev} where $\mathbb{E}\|W(x,v)\|_2^2 \leq C^{\prime}_{\mathcal{D}}$ for arbitary $v \in \Omega_v$. Define $\tilde{h}(x)=\frac{1}{K}\sum_{i=1}^{K}h_i(x)$, we have
	\begin{equation}
	\mathbb{E}[\tilde{h}(x)]=\nabla F(x) \qquad \text{and} \qquad Var[\tilde{h}(x)] \leq \frac{C_{\mathcal{D}}^{\prime}}{KB} +4pd^2l^4_{\mathcal{D}}(\frac{1}{K}+\frac{1}{B}),
	\end{equation}
	so $Var[\tilde{h}(x)]$ can be made arbitrarily small for any $x \in \mathcal{D}$ by making $K$ and $B$ sufficiently large.
\end{Lemma}
\begin{proof}
	First we have 
	\begin{align*}
	\mathbb{E}[\tilde{h}(x)]=\mathbb{E}[{h}_1(x)]=\mathbb{E}[\mathbb{E}[{h}_1(x)|\mathcal{I}]]=&\frac{1}{B}\mathbb{E}_{\mathcal{I}}[\mathbb{E}[\sum_{v_i \in \mathcal{I}}\text{UnibasedGradient}(x,v_i)|\mathcal{I}]] \nonumber \\
	=&\frac{1}{B}\mathbb{E}_{\mathcal{I}}[\sum_{v_i \in \mathcal{I}}\nabla (f_{v_i}(\mathbb{E}_w g_w(x))) ] = \nabla F(x).
	\end{align*}
	Secondly, for any $v \in \Omega_v$, denote $W_i=\text{UnbiasedGradient}(x,v_i)$, $h_v=\nabla (f_{v}(\mathbb{E}_w g_w(x)))$ and  $h({\mathcal{I}})=\mathbb{E}[{h}_1(x)|\mathcal{I}]=\frac{1}{B}\sum_{v_i \in \mathcal{I}}h_{v_i}$, we have
	\begin{align*}
	Var[\tilde{h}(x)] =& \mathbb{E}[Var[\tilde{h}(x)|\mathcal{I}]]+Var[\mathbb{E}[\tilde{h}(x)|\mathcal{I}]] \nonumber\\
	=& \frac{1}{K}\mathbb{E}[Var[{h}_1(x)|\mathcal{I}]]+Var_{\mathcal{I}}[h(\mathcal{I})] \nonumber\\
	= & \frac{1}{K}\mathbb{E}\big[\mathbb{E}[({h}_1(x)-h(\mathcal{I}))^\intercal({h}_1(x)-h(\mathcal{I}))|\mathcal{I}]\big]+\frac{1}{B}Var_{v}[h_v] \nonumber\\ 
	=& \frac{1}{K\cdot B^2}\mathbb{E}\big[\mathbb{E}[(\sum_{i=1}^B W_i-h_{v_i}+h_{v_i}-h(\mathcal{I}))^\intercal(\sum_{i=1}^B W_i-h_{v_i}+h_{v_i}-h(\mathcal{I}))|\mathcal{I}]\big]+\frac{1}{B}Var_{v}[h_v] \nonumber\\
	=& \frac{1}{K\cdot B^2}\mathbb{E}\big[\mathbb{E}[\sum_{i=1}^B\|W_i-h_{v_i}\|_2^2+\sum_{i=1}^B\sum_{j=1}^B(h_{v_i}-h(\mathcal{I}))^\intercal(h_{v_j}-h(\mathcal{I}))|\mathcal{I}]\big]+\frac{1}{B}Var_{v}[h_v] \nonumber\\
	 \leq & \frac{C_{\mathcal{D}}^{\prime}}{KB} +4pd^2l^4_{\mathcal{D}}(\frac{1}{K}+\frac{1}{B})
	\end{align*}
	where the last inequality follow from the definition of $C_{\mathcal{D}}^{\prime}$ and fact that each component of $h_v$ is bounded by $dl^2_{\mathcal{D}}$ for any $v \in \Omega_v$, according to the definition of $l_{\mathcal{D}}$ and $h_v$. The equality above it follows from the independence between the $W_i$'s given $\mathcal{I}$.
\end{proof}

\begin{Theorem}\label{Th3}
	Consider the Simulated SCSG Algorithm 5 with options II. Suppose the setting in Theorem \ref{Big1}. Fix $\epsilon > 0$ as the level of accuracy. Let $\lambda$ is small and $M$ is sufficiently large so that 
	\begin{equation}
	\alpha = \frac{2}{\mu(1-\frac{8}{\mu}C_{\mathcal{D}}\lambda)\lambda M}+\frac{(\frac{8}{\mu}C_{\mathcal{D}}+8L)\lambda}{1-\frac{8}{\mu}C_{\mathcal{D}}\lambda}< 1,
	\end{equation}
	while making one of $K$ and $B$ large enough so that 
	\begin{equation}
	\frac{4(\lambda+\frac{1}{2\mu})}{1-\frac{8}{\mu}C_{\mathcal{D}}\lambda} \cdot Var[\tilde{h}] < \epsilon
	\end{equation}
	Then we have the following result for the Simulated SCSG :
	\begin{equation}
	\mathbb{E}[F(\tilde{x}_{s})-F(x_{\star})] \leq \alpha^s \cdot \mathbb{E}[F(\tilde{x}_{0})-F(x_{\star})]+\frac{1}{1-\alpha} \cdot \epsilon
	\end{equation}
\end{Theorem}
\begin{proof}
	Conditioning on $x_{t-1}$, we can take expectation with respect to $v_t \in \Omega_v$ to obtain 
	\begin{align}\label{vt}
	\mathbb{E}[\|\nu_t\|_2^2 \ | \  x_{t-1}] \leq & 2\mathbb{E}[\|W(x_{t-1},v_t)-W(\tilde{x},v_t)\|_2^2\ | \  x_{t-1}] + 4\| \nabla F(\tilde{x}) \|_2^2 +4 \|\tilde{h}(\tilde{x})-\nabla F(\tilde{x})\|_2^2 \nonumber \\
	\leq & 2C_{\mathcal{D}}\|x_{t-1}-\tilde{x}\|_2^2+ 8L[F(\tilde{x})-F(x_{\star})]+4 \|\tilde{h}(\tilde{x})-\nabla F(\tilde{x})\|_2^2  \nonumber \\
	\leq & 4C_{\mathcal{D}}(\|x_{t-1}-{x_{\star}}\|_2^2 +\|\tilde{x}-x_{\star}\|_2^2) + 8L[F(\tilde{x})-F(x_{\star})]+4 \|\tilde{h}(\tilde{x})-\nabla F(\tilde{x})\|_2^2  \nonumber \\
	\leq & \frac{8}{\mu}C_{\mathcal{D}} \cdot [F(x_{t-1})-F(x_{\star})] + (\frac{8}{\mu}C_{\mathcal{D}} + 8L)\cdot[F(\tilde{x})-F(x_{\star})]+4 \|\tilde{h}(\tilde{x})-\nabla F(\tilde{x})\|_2^2.
	\end{align}
	where the second inequality follows from Theorem \ref{th3} and equation \eqref{svrg8}. The last inequality follows from the strong convexity of $F(\cdot)$. Now following \eqref{vt}, we can write
	\begin{align}
	&\mathbb{E}[\|x_t-x_{\star}\|_2^2 \ | \ x_{t-1}] \nonumber \\
	=&\|x_{t-1}-x_{\star}\|_2^2-2\lambda(x_{t-1}-x_{\star})^{\intercal}\mathbb{E}[\nu_t \ | x_{t-1}]+\lambda^2\mathbb{E}[\|\nu_t\|_2^2 \ | x_{t-1}] \nonumber \\
	\leq & \|x_{t-1}-x_{\star}\|_2^2-2\lambda(x_{t-1}-x_{\star})^{\intercal}(\nabla F(x_{t-1})-\nabla F(\tilde{x})+\tilde{h}(\tilde{x}))\nonumber \\
	&+ \frac{8}{\mu}C_{\mathcal{D}} \lambda^2\cdot [F(x_{t-1})-F(x_{\star})] + (\frac{8}{\mu}C_{\mathcal{D}} + 8L)\lambda^2\cdot[F(\tilde{x})-F(x_{\star})]+4\lambda^2 \|\tilde{h}(\tilde{x})-\nabla F(\tilde{x})\|_2^2 \nonumber \\
	\leq & \|x_{t-1}-x_{\star}\|_2^2-2\lambda [F(x_{t-1})-F(x_{\star})]+ 2\lambda (x_{t-1}-x_{\star})^{\intercal}(\tilde{h}(\tilde{x})-\nabla F(\tilde{x}))\nonumber\\ &+\frac{8}{\mu}C_{\mathcal{D}} \lambda^2\cdot [F(x_{t-1})-F(x_{\star})] + (\frac{8}{\mu}C_{\mathcal{D}} + 8L)\lambda^2\cdot[F(\tilde{x})-F(x_{\star})]+4\lambda^2 \|\tilde{h}(\tilde{x})-\nabla F(\tilde{x})\|_2^2  \nonumber \\
	=& \|x_{t-1}-x_{\star}\|_2^2-2\lambda(1-\frac{4}{\mu}C_{\mathcal{D}}\lambda) [F(x_{t-1})-F(x_{\star})]+(\frac{8}{\mu}C_{\mathcal{D}} + 8L)\lambda^2[F(\tilde{x})-F(x_{\star})] \nonumber \\
	&+4\lambda^2 \|\tilde{h}(\tilde{x})-\nabla F(\tilde{x})\|_2^2+2\lambda (x_{t-1}-x_{\star})^{\intercal}(\tilde{h}(\tilde{x})-\nabla F(\tilde{x})),
	\end{align}
	where the third line follows from the convexity of $F(\cdot)$. Now we consider a fixed stage $s$, so that $\tilde{x}=\tilde{x}_{s-1}$ and $\tilde{x}_s$ is selected uniformly after all $M$ updates are completed. Summing over the previous inequality over $t=1,...,M$, taking expectation and use options II at stage $s$, we obtain
	\begin{align}
	&\mathbb{E}[\|x_{M}-x_{\star}\|_2^2]+2\lambda(1-\frac{4}{\mu}C_{\mathcal{D}}\lambda)M\mathbb{E}[F(\tilde{x}_s)-F(x_{\star})] \nonumber \\
	\leq & \mathbb{E}[\|x_{0}-x_{\star}\|_2^2] +(\frac{8}{\mu}C_{\mathcal{D}} + 8L)\lambda^2M\mathbb{E}[F(\tilde{x})-F(x_{\star})]+4\lambda^2M \|\tilde{h}(\tilde{x})-\nabla F(\tilde{x})\|_2^2+2\lambda M \mathbb{E}[(\tilde{x}_s-x_{\star})^{\intercal}(\tilde{h}(\tilde{x})-\nabla F(\tilde{x}))]\nonumber \\
	\leq & \mathbb{E}[\|x_{0}-x_{\star}\|_2^2] +(\frac{8}{\mu}C_{\mathcal{D}} + 8L)\lambda^2M\mathbb{E}[F(\tilde{x})-F(x_{\star})]+4\lambda M(\lambda+\frac{1}{2\mu}) \|\tilde{h}(\tilde{x})-\nabla F(\tilde{x})\|_2^2+\frac{\mu}{2}\lambda M  \mathbb{E}[\|\tilde{x}_s-x_{\star}\|_2^2] \nonumber\\
	\leq & \mathbb{E}[\|x_{0}-x_{\star}\|_2^2] +(\frac{8}{\mu}C_{\mathcal{D}} + 8L)\lambda^2M\mathbb{E}[F(\tilde{x})-F(x_{\star})]+4\lambda M(\lambda+\frac{1}{2\mu}) \|\tilde{h}(\tilde{x})-\nabla F(\tilde{x})\|_2^2+\lambda M  \mathbb{E}[F(\tilde{x}_s)-F(x_{\star})],
	\end{align}
	where the second inequality follows from $2a^\intercal b\leq {\beta}\|a\|_2^2+\frac{1}{\beta} \|b\|_2^2$ while $\beta=\frac{\mu}{2}$. The last inequality follows from strong convexity of $F(\cdot)$. Finally, taking expectation over the randomness of $\tilde{h}(\tilde{x})$, we have
	\begin{align}
	&\lambda(1-\frac{8}{\mu}C_{\mathcal{D}}\lambda)M\mathbb{E}[F(\tilde{x}_s)-F(x_{\star})]\nonumber \\
	\leq& \mathbb{E}[\|\tilde{x}-x_{\star}\|_2^2] + (\frac{8}{\mu}C_{\mathcal{D}} + 8L)\lambda^2M\mathbb{E}[F(\tilde{x})-F(x_{\star})]+4\lambda M(\lambda+\frac{1}{2\mu}) Var[\tilde{h}(\tilde{x})] \nonumber \\
	\leq & \frac{2}{\mu} \mathbb{E}[F(\tilde{x})-F(x_{\star})]+(\frac{8}{\mu}C_{\mathcal{D}} + 8L)\lambda^2M\mathbb{E}[F(\tilde{x})-F(x_{\star})] +4\lambda M(\lambda+\frac{1}{2\mu}) Var[\tilde{h}(\tilde{x})] \nonumber \\
	= & (\frac{2}{\mu}+(\frac{8}{\mu}C_{\mathcal{D}} + 8L)\lambda^2 M)\mathbb{E}[F(\tilde{x})-F(x_{\star})]+4\lambda M(\lambda+\frac{1}{2\mu}) Var[\tilde{h}(\tilde{x})]
	\end{align}
	Thus we obtain
	\begin{align}
	\mathbb{E}[F(\tilde{x}_s)-F(x_{\star})] &\leq \bigg[\frac{2}{\mu(1-\frac{8}{\mu}C_{\mathcal{D}}\lambda)\lambda M}+\frac{(\frac{8}{\mu}C_{\mathcal{D}}+8L)\lambda}{1-\frac{8}{\mu}C_{\mathcal{D}}\lambda}\bigg]\mathbb{E}[F(\tilde{x}_{s-1})-F(x_{\star})]+ \frac{4(\lambda+\frac{1}{2\mu})}{1-\frac{8}{\mu}C_{\mathcal{D}}\lambda} Var[\tilde{h}(\tilde{x})] \nonumber \\
	&\leq \alpha \cdot \mathbb{E}[F(\tilde{x}_{s-1})-F(x_{\star})] + \epsilon
	\end{align}
	This implies that $\mathbb{E}[F(\tilde{x}_{s})-F(x_{\star})] \leq \alpha^s \cdot \mathbb{E}[F(\tilde{x}_{0})-F(x_{\star})]+\frac{\epsilon}{1-\alpha} $. The conclusion follows.
\end{proof}
\begin{Corollary}
	Let $\{\tilde{x}_s\}_{s\geq 0}$ be the sequence of output from each epoch of the Simulated SCSG algorithm and define $\tilde{y}_s=\min\limits_{t\leq s}\{F(\tilde{x}_t)-F(x_{\star})\}$ for $s \geq 0$ to be the lowest objective value after epoch s. Then, with probability 1, we have $\inf\limits_{s\geq 0}\tilde{y}_s \leq \frac{\epsilon}{1-\alpha}$.
\end{Corollary}
\begin{proof}
It follows from Theorem \ref{Th3} that we can find $0<\alpha<1$ where $\mathbb{E}[F(\tilde{x}_{s})-F(x_{\star})] \leq \alpha^s \mathbb{E}[F(\tilde{x}_{0})-F(x_{\star})]+\frac{\epsilon}{1-\alpha}$. We also have $\sup\limits_{x \in \mathcal{D}}\{F(x)-F(x_{\star})\} \leq 2l_{\mathcal{D}}$ from the definition of $l_{\mathcal{D}}$. It follows that for any $\tilde{x}_0\in \mathcal{D}$, we have that $\mathbb{E}[F(\tilde{x}_{s})-F(x_{\star})|\tilde{x}_0] \leq \alpha^s\cdot 2l_{\mathcal{D}}+\frac{\epsilon}{1-\alpha}$. For any $\rho>0$, pick $N$ large enough  that $\delta=(\alpha^N\cdot 2l_{\mathcal{D}}+\frac{\epsilon}{1-\alpha})(\frac{\epsilon}{1-\alpha}+\rho)^{-1}<1$, we have 
\begin{equation*}
\mathbb{P}(\tilde{y}_N \geq \frac{\epsilon}{1-\alpha}+\rho) \leq \mathbb{P}(F(\tilde{x}_N)-F(x)\geq \frac{\epsilon}{1-\alpha}+\rho)\leq \mathbb{E}[F(\tilde{x}_0)-F(x)](\frac{\epsilon}{1-\alpha}+\rho)^{-1} \leq \delta.
\end{equation*}
However, if we denote $\mathcal{X}_N$ to be the distribution of $\tilde{x}_N$ conditioning on $\tilde{y}_N\geq \frac{\epsilon}{1-\alpha}+\rho$, then it follows from the Markov Property that 
\begin{align*}
\mathbb{P}(\tilde{y}_{2N} \geq \frac{\epsilon}{1-\alpha}+\rho)=&\mathbb{P}(\tilde{y}_{2N} \geq \frac{\epsilon}{1-\alpha}+\rho|\tilde{y}_N \geq \frac{\epsilon}{1-\alpha}+\rho)\mathbb{P}(\tilde{y}_N \geq \frac{\epsilon}{1-\alpha}+\rho)\nonumber\\
=&\mathbb{P}(\min\limits_{N+1 \leq s \leq 2N}\{F(\tilde{x}_s)-F(x_{\star})\} \geq \frac{\epsilon}{1-\alpha}+\rho|\tilde{y}_N \geq \frac{\epsilon}{1-\alpha}+\rho)\mathbb{P}(\tilde{y}_N \geq \frac{\epsilon}{1-\alpha}+\rho) \nonumber\\
=& (\mathbb{P}_{\tilde{x}_N \sim \mathcal{X}_N}\mathbb{P}(\min\limits_{N+1 \leq s \leq 2N}\{F(\tilde{x}_s)-F(x_{\star})\} \geq \frac{\epsilon}{1-\alpha}+\rho|\tilde{x}_N))\cdot\mathbb{P}(\tilde{y}_N \geq \frac{\epsilon}{1-\alpha}+\rho) \nonumber\\
\leq&(\mathbb{P}_{\tilde{x}_N \sim \mathcal{X}_N}\mathbb{P}(F(\tilde{x}_{2N})-F(x_{\star}) \geq \frac{\epsilon}{1-\alpha}+\rho|\tilde{x}_N))\cdot \delta \nonumber\\
\leq & (\mathbb{P}_{\tilde{x}_N \sim \mathcal{X}_N}\mathbb{E}[F(\tilde{x}_{2N})-F(x_{\star})|\tilde{x}_N])\cdot( \frac{\epsilon}{1-\alpha}+\rho)^{-1}\cdot \delta \nonumber\\
=&(\mathbb{P}_{\tilde{x}_0 \sim \mathcal{X}_N}\mathbb{E}[F(\tilde{x}_{N})-F(x_{\star})|\tilde{x}_0])\cdot( \frac{\epsilon}{1-\alpha}+\rho)^{-1}\cdot \delta \nonumber\\
\leq&\mathbb{P}_{\tilde{x}_0 \sim \mathcal{X}_N}(\alpha^N\cdot 2l_{\mathcal{D}}+\frac{\epsilon}{1-\alpha})\cdot( \frac{\epsilon}{1-\alpha}+\rho)^{-1}\cdot \delta \leq \delta^2
\end{align*}
Continue on, we can prove that $\mathbb{P}(\tilde{y}_{kN} \geq \frac{\epsilon}{1-\alpha}+\rho) \leq \delta^k$. Thus if we define the set $\mathcal{A}_{\rho}=\{\inf\limits_{s\geq 0}\tilde{y}_s \geq \frac{\epsilon}{1-\alpha}+\rho\}$ and $\mathcal{A}=\{\inf\limits_{s\geq 0}\tilde{y}_s > \frac{\epsilon}{1-\alpha}\}$ in probability space, we have
\begin{equation}
\mathbb{P}(\mathcal{A}_{\rho})=\mathbb{P}(\inf_{s\geq 0}\tilde{y}_{s} \geq \frac{\epsilon}{1-\alpha}+\rho) \leq \mathbb{P}(\tilde{y}_{kN} \geq \frac{\epsilon}{1-\alpha}+\rho) \leq \delta^k,
\end{equation}
for any $k \geq 1$. Since $\delta<1$, we have $\mathbb{P}(\mathcal{A}_{\rho})=0$ for any $\rho>0$ which implies $\mathbb{P}(\mathcal{A}) =\mathbb{P}( \bigcup\limits_{n \geq 1}\mathcal{A}_{\frac{1}{n}}) \leq \sum\limits_{n=1}^{\infty}\mathbb{P}(\mathcal{A}_{\frac{1}{n}})=0.$
So, with probability 1, $\inf\limits_{s\geq 0}\tilde{y}_s \leq \frac{\epsilon}{1-\alpha}$.
\end{proof}
\section{Numerical Experiments}
\subsection{Cox's partial likelihood}
We implemented the two algorithms on minimizing a regularized Cox's negative partial log- likelihood and compared the performance with the  Compositional-SVRG-1 algorithm in \cite{pmlr-v54-lian17a} and  gradient decent algorithm:
\begin{align}
\min_{\beta \in \mathbb{R}^p} \quad \frac{1}{n} \sum_{i=1}^n \Delta_i[-X_i^\top \beta + \log\{\sum_{j = 1}^n \mathbb{I}(Y_j \geq Y_i) \exp(X_j^\top \beta)\}] + \frac{1}{2}\|\beta\|_2^2, \label{log-pl}.
\end{align}
As in the setting of Examples, $(X_i, Y_i, \Delta_i)$, $i = 1,\ldots, n$ are I.I.D. observations, $X_i \in \mathbb{R}^p$ is the feature vector, $Y_i = \min(T_i, C_i)$ and $\Delta_i = \mathbb{I}\{Y_i = T_i\}$, $T_i$ is the true life time and $C_i$ is the censoring time which is independent of $T_i$. It is easy to see that each component function is strongly convex and has Lipschitz continuous gradients. Our numerical results are based on simulated data and here is our settings. We set $n = 10^4$, $p = 10^3$ and let every entry of $X$ follow I.I.D. standard normal distribution. $T$ is generated by standard exponential base line hazard function and $C$ is independent of $T$ with censoring rate around 30\%.

In the Simulated SVRG algorithm, we set the step size to be $\lambda=0.01$, the number of iterations in the inner loop to be $M=100$ and the base level to be $n_0=0$ whereas in the Simulated SCSG algorithm, we set he step size to be $\lambda=0.0005$, the number of iterations in the inner loop to be $M=100$, the batch size to be $B=100$, number of simulations to be $K=50$ and the base levels to be $n_0=2$.Accordingly, in the compositional-SVRG-1 algorithm, we set the step size to be $\lambda=0.001$, the number of iterations in the inner loop to be $M=100$ and the batch size to be $B=500$ whereas in the gradient descent algorithm, we set the step size to be $\lambda=0.01$. 

The numerical result presented below is a plot of the logarithm of the difference between function value and the optimal value against number of iterations. From this plot, we can see that the proposed algorithms converges linearly to the optimal solution. Algorithm 4, the Simulated SVRG has the best performance amongst the group. This result is expected since the convergence rate of Simulated SVRG does not depend on $m$. However, this is not the case for the compositional-SVRG-1 algorithms. Also, as expected, we can see that Algorithm 5, the Simulated SCGS does not perform as well as Algorithm 4 or compositional-SVRG-1 since it does not involve full gradient evaluations. An interesting finding is that gradient descent algorithm has the worst performance in terms of iteration complexity. 
\begin{figure}[htp]
	\begin{center}
		\includegraphics[scale = 0.5]{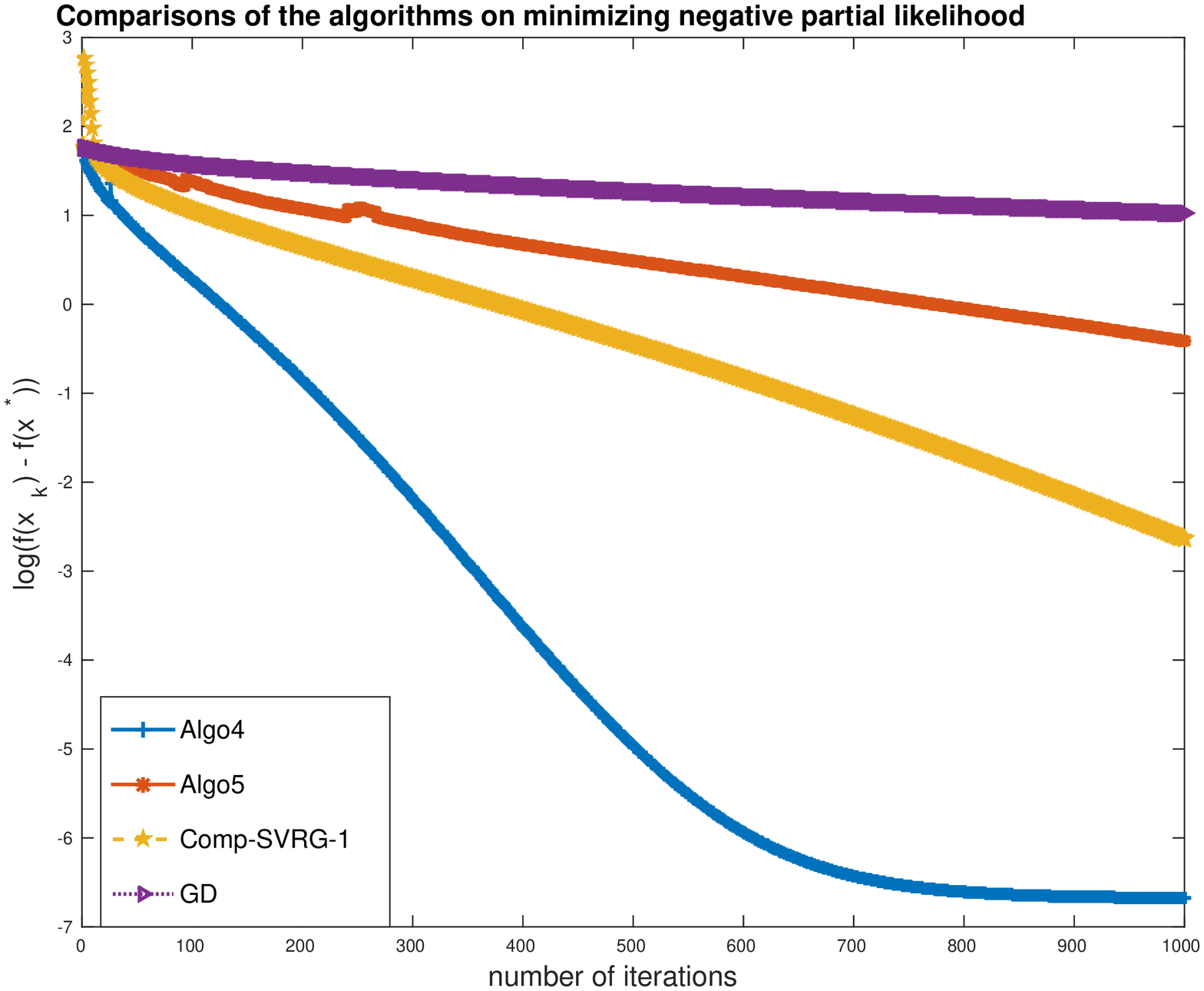}
	\end{center}
	
\end{figure}

\subsection{Conditional Random Fields}

We also implemented the proposed algorithms on the optical character recognition (OCR) dataset to train conditional random field in \cite{taskar2004max}. In contrast to the Cox's partial likelihood problem, the full gradient of CRF can be efficiently computed by dynamic programming method (the Viterbi algorithm\cite{1450960}) as mentioned in Examples. We compare the proposed algorithms with gradient descent. 

In the Simulated SVRG algorithm, we set the step size to be $\lambda=0.001$, the number of iterations in the inner loop to be $M=200$ and the base level to be $n_0=0$.
In the Simulated SCGS descent algorithm, we set the step size to be $\lambda=0.0001$, the number of iterations in the inner loop to be $M=200$, the batch size to be $B=100$, number of simulations to be $K=10$ and the base levels to be $n_0=2$. Finally, in the gradient descent algorithm, we set the step size to be $\lambda=0.01$. 

The plot below is the logarithm of the difference between function value and the optimal value against number of iterations. Similarly as before, Algorithm 4, the Simulated SVRG has the best performance amongst the group. However, in this example, gradient descent algorithm actually outperforms Algorithm 5, Simulated SVRG in terms iteration complexity. As mentioned before, this is due to the lack of full gradient evaluation in Algorithm 5 which, as we can see, has a more negative effect when the dataset is large. This is expected because the larger the dataset, the more costly it would be to obtain accurate gradient estimation. Another interesting observation is that sometimes we can see function value actually increase after iteration. This is due to the variance of the gradient estimate.

\begin{figure}[tbp]
	\begin{center}
		\includegraphics[scale = 0.5]{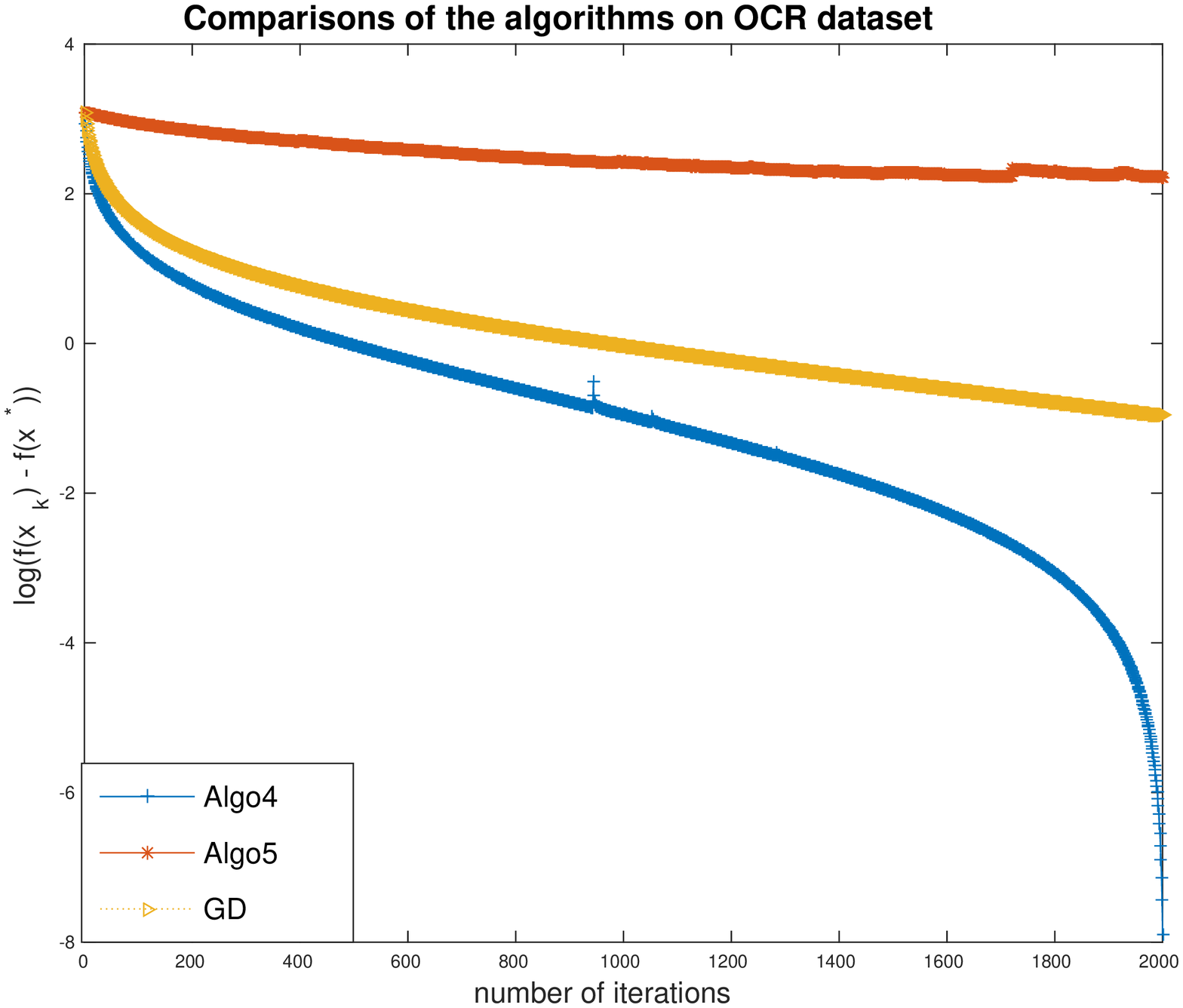}
	\end{center}
	
\end{figure}

\bibliographystyle{unsrt}
\bibliography{hw1}
\end{document}